\documentclass[11pt,a4paper]{article}
\usepackage[top=3cm, bottom=3cm, left=3cm, right=3cm]{geometry}
\usepackage[latin1]{inputenc}
\usepackage{amsmath}
\usepackage{amsthm}
\usepackage{amsfonts}
\usepackage{amssymb}
\usepackage{graphics}

\usepackage[hang,flushmargin]{footmisc} 

\usepackage{amsmath,amsthm,amssymb,graphicx, multicol, array}
\usepackage{enumerate}
\usepackage{enumitem}
\usepackage{xcolor}

\usepackage[bookmarks,colorlinks,breaklinks, pagebackref]{hyperref}

\hypersetup{linkcolor=blue,citecolor=red,filecolor=blue,urlcolor=blue} 

\usepackage{tabto}

\numberwithin{equation}{section}

\usepackage{tikz}
\usepackage{pgfplots}

\DeclareMathOperator{\lcm}{lcm}

\newtheorem{thm}{Theorem}[section]
\newtheorem{lem}{Lemma}[section]

\newtheorem{hyp}{Hypothesis}[section]

\newtheorem{prob}{Problem}[section]
\newtheorem{cor}{Corollary}[section]
\newtheorem{dfn}{Definition}[section]

\newtheorem{rmk}{Remark}[section]

\newcommand{\N}{\mathbb{N}}
\newcommand{\Z}{\mathbb{Z}}

\newcommand{\R}{\mathbb{R}}
\newcommand{\C}{\mathbb{C}}

\usepackage{fancyhdr}
\title{Note on the Chowla Conjecture and the Discrete Fourier Transform}
\date{}
\author{N. A. Carella}

\begin{document}
\maketitle

\begin{abstract}
Let $x\geq 1$ be a large prime number, and let $\mu(n)\in\{-1,0,1\}$ be the Mobius function. This note shows that the average value of the autocorrelation over $t\ne0$ satisfies 
$x^{-1}\sum_{1\leq t\leq x,}\sum_{n \leq x} \mu(n) \mu(n+t) =O\left( x(\log x)^{-c}\right)$. Furthermore, 
autocorrelation of the periodic Mobius function of period $x$ has an asymptotic formula of the form $\sum_{n \leq x} \mu(n+a_0) \mu(n+a_1)\cdots\mu(n+a_{k-1}) =O\left( x(\log x)^{-c}\right)$, where $a_0<a_1<\cdots<a_{k-1}<x$ is an integer $k$-tuple, $c>0$ is an arbitrary constant. \let\thefootnote\relax\footnote{\today \date{} \\
	\textit{AMS MSC}: Primary 11N37, Secondary 11L03, 11K31. \\
	\textit{Keywords}: Arithmetic function; Mobius function; Liouville function; Autocorrelation function; Correlation function; Chowla conjecture; Discrete Fourier transform.}
\end{abstract}

\tableofcontents

\section{Introduction} \label{S5757}
The symbols $\mathbb{N} = \{ 0, 1, 2, 3, \ldots \}$ and $\mathbb{Z} = \{ \ldots, -3, -2, -1, 0,  1, 2, 3, \ldots \}$ denote the subsets of integers. For $n \in \mathbb{N}$, the Mobius function $\mu:\mathbb{N} \longrightarrow \{-1,0,1\}$ is defined by
\begin{equation}\label{eq5757.100A}
	\mu(n) =
	\left \{
	\begin{array}{ll}
		(-1)^{w}     &n=p_1 p_2 \cdots p_w\\
		0           &n \ne p_1 p_2 \cdots p_w,\\
	\end{array}
	\right .
\end{equation}
where the $p_i\geq 2$ are primes. The Chowla conjecture is a quantitative statement on the asymptotic formula for the arithmetic average
\begin{equation} \label{eq5757.200A}
\sum_{n \leq x} \mu(n+a_0) \mu(n+a_1)\cdots \mu(n+a_{k-1}) =o(x), 
\end{equation}	
where $a_i\ne a_j$ for $i\ne j$, see \cite{CS1965}, \cite{RO2018}. 
This problem has an extensive literature, and there are sharper conjectures such as
\begin{equation} \label{eq5757.205A}
\sum_{n \leq x} \mu(n+a_0) \mu(n+a_1)\cdots \mu(n+a_{k-1}) =O\left(\sqrt{x\log \log x} \right) , 
\end{equation}	
as explicated in \cite[Theorem 1.3]{KS2022}. The very best estimates in the literature, for $k=2$, have asymptotic formulae for the logarithmic average of the forms
\begin{equation} \label{eq5757.210A}
\sum_{n \leq x} \frac{\mu(n) \mu(n+t)}{ n}=O \left (\frac{\log x}{\sqrt{ \log \log x }} \right ),
\end{equation} 
or weaker, where $t \ne0$ is a fixed integer, see \cite[Corollary 1.5]{HR2021}. A short calculation shows that \eqref{eq5757.210A} implies the arithmetic average
\begin{equation} \label{eq5757.210}
\sum_{n \leq x} \mu(n) \mu(n+t)=O \left (\frac{ x}{\sqrt{ \log \log x }} \right ),
\end{equation} 
see Theorem \ref{thm4742MN.321a}. \\

The first two results within deal with the average value of the autocorrelation function \eqref{eq5757.210}. The error term of the asymptotic formula for the uniformly weighted average in Theorem \ref{thm5757MA.555A} is slightly smaller than the error term of the asymptotic formula for the harmonic weighted average in Theorem \ref{thm1279MA.555B} in Section \ref{S1279MA}. Both of these asymptotic formulas prove that the autocorrelation function exhibits sufficient cancellations for any $t\ne0$ on average.

\begin{thm} \label{thm5757MA.555A} Let $\mu:\mathbb{N} \longrightarrow \{-1,0,1\}$ be the Mobius function. If $x\geq1$ is a large integer, and $t\leq x$ is an integer, then the uniformly weighted average satisfies 
	\begin{equation} \label{eq1279VM.400B}
		\frac{1}{x}\sum_{1\leq t\leq x,}	\sum_{n \leq x} \mu(n) \mu(n+t) =O\left( \frac{x}{(\log x)^{2c}}\right)\nonumber, 
	\end{equation}	
	where $c>0$ is an arbitrary constant.
\end{thm}

The results for the autocorrelation function on average given in Theorem \ref{thm5757MA.555A} and Theorem \ref{thm1279MA.555B} are significantly sharper than the very best in the current literature. The very best result in the literature for the uniform weighted average over the interval $[1,x]$ has something of the form
	\begin{equation} \label{eq1279VM.400Z}
	\frac{1}{x}\sum_{1\leq t\leq x,}		\sum_{n \leq x} \mu(n) \mu(n+t) =O\left(x\frac{\log\log x}{\log x}\right)\nonumber, 
\end{equation}	
confer \cite[Theorem 1]{MR2015} for the finer details, which is remarkably weaker.

\begin{thm} \label{thm5757MN.321} Let $x$ be a large prime number, and let $\mu:\mathbb{N} \longrightarrow \{-1,0,1\}$ be the periodic Mobius function. If $t\ne0$ is an integer, then, 
	\begin{equation} \label{eq5757MN.321}
	\sum_{n \leq x} \mu(n) \mu(n+t) =O\left( \frac{x}{(\log x)^{2c}}\right)\nonumber, 
	\end{equation}	
	where $c>0$ is an arbitrary constant.
\end{thm}	
The more general version of the $k$-correlation function, stated below, can be computed recursively. 
\begin{thm} \label{thm5757MN.121} Let $x>1$ be a large prime number. Assume that the Mobius function $\mu:\mathbb{N} \longrightarrow \{-1,0,1\}$ is extended to a periodic function of period $x$ on the set of integers. If $a_0< a_1<\cdots<a_{k-1}<x$ is an integer $k$-tuple then, 
\begin{equation} \label{eq5757MN.121}
\sum_{n \leq x} \mu(n+a_0) \mu(n+a_1)\cdots \mu(n+a_{k-1}) =O\left( \frac{x}{(\log x)^{2c}}\right)\nonumber, 
\end{equation}	
where $c>0$ is an arbitrary constant.
\end{thm}	

Precisely the same result is valid for the Liouville function. Except for minor notational changes, and implied constant, everything is exactly the same. 

\begin{thm} \label{thm5757LN.121} Let $x>1$ be a large prime number. Assume that the Liouville function $\lambda:\mathbb{N} \longrightarrow \{-1,1\}$ is extended to a periodic function of period $x$ on the set of integers. If $a_0<a_1<\cdots<a_{k-1}<x$ is an integer $k$-tuple then, 
	\begin{equation} \label{eq5757LN.121}
		\sum_{n \leq x} \lambda(n+a_0) \lambda(n+a_1)\cdots \lambda(n+a_{k-1}) =O\left( \frac{x}{(\log x)^{2c}}\right)\nonumber, 
	\end{equation}	
	where $c>0$ is an arbitrary constant.
\end{thm}	

The proofs of Theorem \ref{thm5757MA.555A}, and Theorem \ref{thm5757MN.321}, based on standard results in analytic number theory, and the discrete Fourier transform, are given in Section \ref{S1279MA} and Section \ref{S4747MN} respectively. Similarly, the proofs of Theorem \ref{thm5757MN.121}, and Theorem \ref{thm5757LN.121} are given in Section \ref{S4747K}. \\
 
Probably, the same technique can be extended to the more general finite sum
\begin{equation} \label{eq5757MN.131}
\sum_{n \leq x} \mu(n+a) \mu(f(n)), 
\end{equation}	
where $f(t)\in\Z[t]$ is an irreducible quadratic polynomial.

\section{Discrete Fourier Transform}\label{S8383DFT}
\begin{dfn}\label{dfn8383DFT.100}{\normalfont 
		Let $x\geq 1$ be an integer, and let $f:\N\longrightarrow \C$ be a function. The discrete Fourier transform and its inverse are defined by
		\begin{equation}\label{eq8383DFT.100A}
			\hat{f}(s)=\sum_{1\leq n\leq x}f(n)e^{i2\pi sn/x}
		\end{equation}
		and 
		\begin{equation}\label{eq8383DFT.100B}
			f(n)=\frac{1}{x}\sum_{1\leq s\leq x}\hat{f}(s)e^{-i2\pi sn/x},
		\end{equation}
		respectively.		
	}
\end{dfn} 
Some details on discrete Fourier transform are discussed in \cite[Section 3.11v]{DLMF}.
\begin{thm} \label{thm8383DFT.900} {\normalfont (Linear Convolution Theorem)} Let $f,g: \N \longrightarrow\C$ be a pair of functions. 
	Then, discrete Fourier transform of the correlation function $R(t)=\sum_{n\leq x}f(n)g(n+t)$ is equal to the product of the individual discrete Fourier transforms up to a shift:
	\begin{equation}\label{eq8383DFT.910}
		\widehat{R}(s)=\sum_{1\leq t\leq x}R(t)e^{i2\pi stn/x}=\hat{f}(-s)\hat{g}(n,s).
	\end{equation}	
	
\end{thm}
\begin{proof}[\textbf{Proof}] Use the definition of the discrete Fourier transform of the correlation function to compute it:
	\begin{eqnarray}\label{eq8383DFT.920}
		\widehat{R}(s)&=&\sum_{1\leq t\leq x}\left (\sum_{1\leq n\leq x}f(n)g(n+t)\right )e^{i2\pi st/x}\\
		&=&\sum_{1\leq n\leq x,}\sum_{n+1\leq m<x+n}f(n)g(m)e^{i2\pi s(m-n)/x}\nonumber\\
		&=&\sum_{1\leq n<x}f(n)e^{-i2\pi sn/x}\sum_{n+1\leq m<x+n}g(m)e^{i2\pi sm/x}\nonumber\\
		&=&\hat{f}(-s)\hat{g}(n,s)\nonumber	.
	\end{eqnarray}	
	Conversely, the product $\hat{f}(-s)\hat{g}(n,s)$ is the discrete Fourier transform of the correlation function $R(t)$. 
\end{proof}
Since the factor $\hat{g}(n,s)$ is dependent on $n\leq x$, the precise calculation of the spectrum requires a vector 
\begin{equation}\label{eq1279VM.200C}
	\mu(1), \mu(2),\ldots, \mu(x),\mu(x+1), \mu(x+2),\ldots, \mu(2x)
\end{equation}
of $2x$ values. If the a vector of Mobius function values is extended to a periodic function
\begin{equation}\label{eq1279VM.210C}
	\mu(1), \mu(2),\ldots, \mu(x),\mu(1), \mu(2),\ldots, \mu(x),\mu(1), \mu(2),\ldots, \mu(x),\ldots,
\end{equation}
then, the dependence on the variable $n$ in
\begin{equation}\label{eq8383DFT.900C}
	\hat{g}(n,s)=\sum_{n+1\leq m<x+n}g(m)e^{i2\pi sm/x}	
\end{equation} 
is not present in the circular convolution.
\begin{thm} \label{thm8383DFT.900C} {\normalfont (Circular Convolution Theorem)} Let $f,g: \N \longrightarrow\C$ be a pair of functions. 
	Then, the discrete Fourier transform of the correlation function $R(t)=\sum_{n\leq x}f(n)g(n+t)$ is equal to the product of the individual discrete Fourier transforms:
	\begin{equation}\label{eq8383DFT.910C}
		\widehat{R}(s)=\sum_{1\leq t\leq x}R(t)e^{i2\pi st/x}=\hat{f}(-s)\hat{g}(s).
	\end{equation}	
\end{thm}
\begin{proof}[\textbf{Proof}] The summation index is reduced modulo $x$. Thus,  \begin{equation}\label{eq8383DFT.920C}
		\hat{g}(n,s)=\sum_{n+1\leq m\leq x+n}g(m)e^{i2\pi sm/x}	=\sum_{1\leq m<x}g(m)e^{i2\pi sm/x}=\hat{g}(s)
	\end{equation}
as claimed.
\end{proof}
\subsection{Results for the Liouville and Mobius Functions}\label{S2222}
Some standard results required in the proofs of the correlations functions are recorded in this section. The Mobius function is defined in \eqref{eq5757.100B}. Similarly, for $n \in \mathbb{N}$, the Liouville function $\lambda:\mathbb{N} \longrightarrow \{-1,1\}$ is defined by
\begin{equation}\label{eq5757.100B}
	\lambda(n) =(-1)^{a_1+a_2+\cdots+a_w},
\end{equation}
where $n=p_1^{a_1} p_2^{a_2} \cdots p_w^{a_w}$, the $p_i\geq 2$ are primes, and $a_k\geq0$.

\subsection{Average Orders of Mobius Functions}\label{S2222MN}
\begin{thm} \label{thm2222.500} If $\mu: \N\longrightarrow \{-1,0,1\}$ is the Mobius function, then, for any large number $x\geq1$, the following statements are true.
	\begin{enumerate} [font=\normalfont, label=(\roman*)]
\item $\displaystyle \sum_{n \leq x} \mu(n)=O \left (xe^{-c\sqrt{\log x}}\right )$, \tabto{8cm} unconditionally,
\item $\displaystyle \sum_{n\leq x}\frac{\mu(n)}{n}=O\left( e^{-c\sqrt{\log x}} \right ), $ \tabto{8cm} unconditionally,
	\end{enumerate}where $c>0$ is an absolute constant.
\end{thm}
\begin{proof}[\textbf{Proof}]  See \cite[p.\ 6]{DL2012}, \cite[p.\ 182]{MV2007}, et alii.   
\end{proof}

\subsection{Equidistribution Theorems}\label{S2225P}
There are several mean values and equidistribution results for arithmetic functions over arithmetic progressions of level of distribution $\theta<1/2$. The best known case is the Bombieri-Vinogradov theorem, see \cite[Theorem 15.4]{DL2012}, the case for the Mobius function is proved in \cite[Theorem 1]{WD1973} and \cite{SW1971} states the following.

\begin{cor} {\normalfont (\cite[Corollary 1]{SW1971} }\label{cor2225P.550}  Let $a\geq 1$ be a fixed parameter, and let $x \geq 1$ be a large number. If $C>0$ is a constant, then
	\begin{equation}
		\sum_{q\leq x^{1/2}/\log^B x} \max_{a \bmod q}\max_{z \leq x}  \bigg |\sum_{\substack{n \leq z \\
				n \equiv a \bmod q}} \mu(n) \bigg |\ll \frac{x}{(\log x)^{C}},
	\end{equation}where the constant $B>0$ depends on $C$.
\end{cor}

However, there is no comparable literature on the mean values and equidistribution for arithmetic functions over arithmetic progressions of shifted prime. A comparable result is expected to hold.

\begin{hyp} \label{hyp2225P.550}  Let $a\geq 1$ be a fixed parameter, and let $x \geq 1$ be a large number. If $C>0$ is a constant, then
	\begin{enumerate} [font=\normalfont, label=(\roman*)]
		\item $\displaystyle \sum_{q\leq x^{1/2}/\log^B x} \max_{a \bmod q}\max_{z \leq x}  \bigg |\sum_{\substack{p \leq z \\
				p \equiv d \bmod q}} \lambda(p+a) \bigg |\ll \frac{x}{(\log x)^{C}},$
		\item $\displaystyle \sum_{q\leq x^{1/2}/\log^B x} \max_{a \bmod q}\max_{z \leq x}  \bigg |\sum_{\substack{p \leq z \\
				p \equiv d \bmod q}} \mu(p+a) \bigg |\ll \frac{x}{(\log x)^{C}}, $ 
	\end{enumerate}
	where the constant $B>0$ depends on $C$.
\end{hyp}

A similar mean value for the autocorrelation over arithmetic progression is of interest in the theory of arithmetic correlation functions.
\begin{hyp} \label{hyp2225P.560}  Let $a,b\in \Z$ be a fixed pair of small integers such that $a\ne b$, and let $x \geq 1$ be a large number. If $C>0$ is a constant, then
	\begin{enumerate} [font=\normalfont, label=(\roman*)]
		\item $\displaystyle \sum_{q\leq x^{1/2}/\log^B x} \max_{d \bmod q}\max_{z \leq x}  \bigg |\sum_{\substack{p \leq z \\
				p \equiv d \bmod q}} \lambda(p+a) \lambda(p+b)\bigg |\ll \frac{x}{(\log x)^{C}},$
		\item $\displaystyle \sum_{q\leq x^{1/2}/\log^B x} \max_{d \bmod q}\max_{z \leq x}  \bigg |\sum_{\substack{p \leq z \\
				p \equiv d \bmod q}} \mu(p+a) \mu(p+b)\bigg |\ll \frac{x}{(\log x)^{C}}, $ 
	\end{enumerate}
	where the constant $B>0$ depends on $C$.
\end{hyp}

\subsection{Twisted Exponential Sums}
One of the earliest result for exponential sum with multiplicative coefficients is stated below.
\begin{thm} \label{thm3970ME.300} {\normalfont (\cite{DH1937})} If $\alpha\ne0$ is a real number, and $c>0$ is an arbitrary constant, then
	\begin{enumerate} [font=\normalfont, label=(\roman*)]
	\item $\displaystyle \sup_{\alpha\in\R}\sum_{n \leq x} \mu(n)e^{i 2 \pi  \alpha n}<\frac{c_1x}{(\log x)^{c}}$, \tabto{8cm} unconditionally,
	\item $\displaystyle \sup_{\alpha\in\R}\sum_{n \leq x} \frac{\mu(n)}{n}e^{i 2 \pi  \alpha n}<\frac{c_2}{(\log x)^{c}}, $ \tabto{8cm} unconditionally,
\end{enumerate}
where $c_1=c_1(c)>0$ and $c_2=c_2(c)>0$ are constants depending on $c$, as the number $x \to \infty$.
\end{thm}
The same results are also valid for the Liouville function.
\begin{thm} \label{thm3970LE.300} {\normalfont (\cite{DH1937})} If $\alpha\ne0$ is a real number, and $c>0$ is an arbitrary constant, then
	\begin{enumerate} [font=\normalfont, label=(\roman*)]
		\item $\displaystyle \sup_{\alpha\in\R}\sum_{n \leq x} \lambda(n)e^{i 2 \pi  \alpha n}<\frac{c_3x}{(\log x)^{c}}$, \tabto{8cm} unconditionally,
		\item $\displaystyle \sup_{\alpha\in\R}\sum_{n \leq x} \frac{\lambda(n)}{n}e^{i 2 \pi  \alpha n}<\frac{c_4}{(\log x)^{c}}, $ \tabto{8cm} unconditionally,
	\end{enumerate}
	where $c_3=c_3(c)>0$ and $c_4=c_4(c)>0$ are constants depending on $c$, as the number $x \to \infty$.
\end{thm}

Advanced, and recent works on these exponential sums with multiplicative coefficients, and the more general exponential sums
\begin{equation}\label{S2222FN.100}
	\sum_{n \leq x} f(n)e^{i 2 \pi  \alpha n}
\end{equation}
where $f:\N\longrightarrow \C$ is a function, 
are covered in \cite{MV1977}, \cite{HS1987}, \cite{BH1991}, \cite{MS2002}, et alii.

\subsection{Average Orders of Liouville Functions}\label{S2222LN}
\begin{thm} \label{thm2222LN.500} If $\lambda: \N\longrightarrow \{-1,1\}$ is the Liouville function, then, for any large number $x\geq1$, the following statements are true.
	\begin{enumerate} [font=\normalfont, label=(\roman*)]
		\item $\displaystyle \sum_{n \leq x} \lambda(n)=O \left (xe^{-c\sqrt{\log x}}\right )$, \tabto{8cm} unconditionally,
		\item $\displaystyle \sum_{n\leq x}\frac{\lambda(n)}{n}=O\left( e^{-c\sqrt{\log x}} \right ), $ \tabto{8cm} unconditionally,
	\end{enumerate}where $c>0$ is an absolute constant.
\end{thm}

\subsection{Double Twisted Exponential Sums}
Various results on discrete Fourier transform are effective in producing estimates of the autocorrelation of arithmetic functions. The specific case of the Mobius function is estimated here.
\begin{thm} \label{thm3970M.900} If $x\geq 1$ is a large integer, then
\begin{enumerate} [font=\normalfont, label=(\roman*)]
	\item $\displaystyle \sum_{s \leq x,} \sum_{n \leq x} \mu(n)\mu(n+s)e^{i 2 \pi   ns/x}\ll\frac{x^2}{(\log x)^{2c}}$, \tabto{10cm} unconditionally,
	\item $\displaystyle \sum_{s \leq x,}\sum_{n \leq x} \frac{\mu(n)\mu(n+s)}{n}e^{i 2 \pi   ns/x}\ll\frac{x}{(\log x)^{2c}}, $ \tabto{10cm} unconditionally,
\end{enumerate}where $c>0$ is an absolute constant.
\end{thm}

\begin{proof}[\textbf{Proof}] (i) Let $x\geq1$ be a large integer, and let $s\leq  x$. The discrete Fourier transform of the functions $f(n)=\mu(n)$ and $g(n)=\mu(n+t)$ are
	\begin{equation}\label{eq3970M.910} 
	\hat{f}(s)=	\sum_{n \leq x} \mu(n)e^{i 2 \pi  ns/x}\ll\frac{x}{(\log x)^{c}},
	\end{equation}
	and 	 
	\begin{equation}\label{eq3970M.920} 
\hat{g}(s)=	\sum_{k \leq x} \mu(k)e^{i 2 \pi  ks/x}\ll\frac{x}{(\log x)^{c}},
\end{equation}
these follow from Theorem \ref{thm3970ME.300}. Rearranging the product yields,
\begin{eqnarray}\label{eq3970M.930}
	\hat{f}(-s)\cdot \hat{g}(s)&=&\sum_{n\leq x}f(n) e^{-i2\pi ns/x}\cdot \sum_{k\leq x}g(k) e^{i2\pi ks/x}\\
	&=&\sum_{n\leq x,}\sum_{k\leq x}f(n)g(k)e^{i2\pi (k-n)s/x}\nonumber\\
	&=&\sum_{m\leq x}\left( \sum_{n\leq x}f(n)g(n+m)\right) e^{i2\pi ms/x}\nonumber	.
\end{eqnarray}		
This is equivalent to the convolution theorem, see Theorem \ref{thm8383DFT.900}. Thus, it follows that
\begin{equation}\label{eq3970M.940} 
\sum_{s \leq x,}	\sum_{n \leq x} \mu(n)\mu(n+s)e^{i 2 \pi   ns/x}=\hat{f}(s)\hat{g}(s)\ll\frac{x^2}{(\log x)^{2c}} \nonumber,
\end{equation}
as claimed. The proof of (ii) is similar to the previous one.
\end{proof}

\section{Integers in Arithmetic Progressions}\label{A2002}
An effective asymptotic formula for the number of integers in arithmetic progressions is derived in Lemma \ref{lemA2002.400W}. The derivation is based on a version of the basic large sieve inequality stated below.
\begin{thm}\label{thmA2002.200W} Let $x$ be a large number and let $Q
	\leq x$. If $\{a_n:n\geq1\}$ is a sequence of real number, then
	\begin{equation}\label{eqA2002.100W}
		\sum_{q\leq Q}	q\sum_{1\leq a\leq q}\bigg |\sum_{\substack{n \leq x\\ n\equiv a \bmod q}}a_n-\frac{1}{q}\sum_{n \leq x}a_n\bigg|^2\leq Q\left(10Q+2\pi x \right) \sum_{n \leq x}|a_n|^2\nonumber.
	\end{equation}
\end{thm}
\begin{proof}[\textbf{Proof}] The essential technical details are covered in \cite[Chapter 23]{DH2000}. This inequality is discussed in \cite{GP1967} and the literature in the theory of the large sieve. 
\end{proof}
\begin{lem} \label{lemA2002.400W} If $x \geq 1$ is a large number and $1\leq a< q \leq x$, then
	\begin{equation}\label{eqA2002.400W}
		\max_{1\leq a\leq q}\bigg |\sum_{\substack{n \leq x\\ n\equiv a \bmod q}}1-\frac{1}{q}\sum_{n \leq x}1\bigg|=O\left(\frac{x}{q}e^{-c\sqrt{\log x} }\right),
	\end{equation}
	where $ c>0$ is a constant. In particular,
	\begin{equation}\label{eqA2002.405W}
		\sum_{\substack{n \leq x\\ n\equiv a \bmod q}}1=\left[\frac{x}{q}\right]+O\left(\frac{x}{q}e^{-c\sqrt{\log x} }\right).
	\end{equation}
	
\end{lem}
\begin{proof}[\textbf{Proof}] Trivially, the basic finite sum satisfies the asymptotic \begin{equation}\label{eqA2002.410W}
		\sum_{n \leq x}1=[x]= x-\{x\},
	\end{equation}
	where $[x]=x-\{x\}$ is the largest integer function, and the number of integers in any equivalent class satisfies the asymptotic formula
	\begin{equation}\label{eqA2002.415W}
		\sum_{\substack{n \leq x\\ n\equiv a \bmod q}}1	=	\frac{x}{q}+E(x).
	\end{equation}
	Let $Q=x$ and let the sequence of real numbers be $a_n=1$ for  $n\geq1$. Now suppose that the error term is of the form 
	\begin{equation}\label{eqA2002.420W}
		E(x)=E_0(x)=O\left(x^{\alpha}\right),
	\end{equation}
	where $ \alpha\in(0,1]$ is a constant. Then, the large sieve inequality, Theorem \ref{thmA2002.200W}, yields the lower bound
	\begin{eqnarray}\label{eqA2002.430W}
		\sum_{q\leq x}	q\sum_{1\leq a\leq q}\bigg |\sum_{\substack{n \leq x\\ n\equiv a \bmod q}}1-\frac{1}{q}\sum_{n \leq x}1\bigg|^2&=&\sum_{q\leq x}	q\sum_{1\leq a\leq q}\bigg |\left[\frac{x}{q}\right]+O\left(x^{\alpha}\right)-\frac{[x]}{q}\bigg|^2\nonumber\\[.2cm]
		&=&\sum_{q\leq x}	q\sum_{1\leq a\leq q}\bigg |\frac{x}{q}-\left\{\frac{x}{q}\right\}+O\left(x^{\alpha}\right)-\frac{x-\{x\}}{q}\bigg|^2\nonumber\\[.2cm]
		&\gg&\sum_{q\leq x}	q\sum_{1\leq a\leq q}\bigg |x^{\alpha}+\frac{\{x\}}{q}-\left\{\frac{x}{q}\right\}\bigg|^2\nonumber\\[.2cm]
		&\gg&\sum_{q\leq x}	q\sum_{1\leq a\leq q}\left |x^{\alpha}\right|^2\nonumber\\[.2cm]
		&\gg&x^{2\alpha}\sum_{q\leq x}q\sum_{1\leq a\leq q}1\nonumber\\[.2cm]
		&\gg&x^{2\alpha}\sum_{q\leq x}q^2\nonumber\\[.2cm]
		&\gg&	x^{3+2\alpha} .
	\end{eqnarray}
	On the other direction, it yields the upper bound
	\begin{eqnarray}\label{eqA2002.440W}
		\sum_{q\leq x}	q\sum_{1\leq a\leq q}\bigg |\sum_{\substack{n \leq x\\ n\equiv a \bmod q}}1-\frac{1}{q}\sum_{n \leq x}1\bigg|^2
		&\leq&	 Q\left(10Q+2\pi x \right) \sum_{n \leq x}|a_n|^2\\
		&\leq&	 x\left(10x+2\pi x \right) \sum_{n \leq x}|1|^2\nonumber\\
		&\ll&	 x^3\nonumber.
	\end{eqnarray}
	Clearly, the lower bound in \eqref{eqA2002.430W} contradicts the upper bound in \eqref{eqA2002.440W}. Similarly, the other possibilities for the error term
	\begin{equation}\label{eqA2002.445W}
		E_1=O\left(\frac{x}{(\log x)^c} \right)\quad \text{ and }\quad
		E_2=O\left(xe^{-c\sqrt{\log x} }\right),
	\end{equation}
	contradict large sieve inequality. Therefore, the error term is of the form
	\begin{equation}\label{eqA2002.450W}
		E(x)=O\left(\frac{x}{q}e^{-c\sqrt{\log x} }\right)=O\left(\frac{x}{q(\log x)^c} \right)=O\left(\frac{x}{q }\right),
	\end{equation}
	where $ c>0$ is a constant.
\end{proof}

\section{Averages and Norms} \label{S1279MA}
The verifications of Theorem \ref{thm5757MA.555A} and Theorem \ref{thm1279MA.555B} are similar. Nevertheless, both proofs are included. \\
\subsection{Averages of the Autocorrelation Functions}
Given any large integer $x$, the vector of Mobius function values \begin{equation}\label{eq1279VM.200}
	\mu(1), \mu(2),\ldots, \mu(x)
\end{equation} is extended to a periodic function. This technique allows the introduction of a combination of discrete Fourier transform analysis and analytic number theory to compute an effective upper bound of the average value of the autocorrelation function.

\begin{proof}[{\bfseries Proof of Theorem {\normalfont \ref{thm5757MA.555A}}}.] Suppose that $x\geq1$ is a large integer, and consider the autocorrelation function 
	\begin{equation}\label{eq1279VM.410B}
		R(t)=\sum_{n \leq x} \mu(n) \mu(n+t),	
	\end{equation}	
	and its spectrum function
	\begin{equation}\label{eq1279VM.420B}
		\hat{R}(s)=\sum_{1\leq t\leq x}\left( \sum_{1\leq n \leq x} \mu(n) \mu(n+t)\right) e^{i2\pi st/x} .	
	\end{equation}		
	
	The discrete Fourier transform of the autocorrelation function decomposes as a product of two factors:
	\begin{eqnarray}\label{eq1279VM.430B}
		\hat{R}(s)
		&=&\sum_{n \leq x} \mu(n) e^{-i2\pi sn/x}\sum_{n+1<m<x+n} \mu(m)e^{i2\pi sm/x}\\
		&=&\hat{f}(-s)\hat{f}(n,s)\nonumber,	
	\end{eqnarray}
	see the Convolution theorem in Theorem \ref{thm8383DFT.900}. \\
	
	Inverting the discrete Fourier transform in \eqref{eq1279VM.430B}, see Definition \ref{dfn8383DFT.100}, yields
	\begin{eqnarray}\label{eq1279VM.440B}
		\sum_{1\leq n \leq x} \mu(n) \mu(n+t)&=&\frac{1}{x}	\sum_{1\leq s\leq x}\hat{R}(s)e^{-i2\pi st/x}\\
		&=&\frac{1}{x}	\sum_{1\leq s\leq x}\hat{f}(-s)\hat{f}(n,s)e^{-i2\pi st/x}\nonumber.
	\end{eqnarray}	
	Next, averaging over the parameter $t\leq x$, yields
	\begin{eqnarray}\label{eq1279VM.450B}
		\frac{1}{x}\sum_{1\leq t \leq  x,}	\sum_{1\leq n \leq x} \mu(n) \mu(n+t)&=&\frac{1}{x}	\sum_{1\leq t\leq  x}\left (\frac{1}{x}	\sum_{1\leq s\leq x}\hat{f}(-s)\hat{f}(n,s)e^{-i2\pi st/x}\right )\nonumber\\
		&=&\frac{1}{x^2}	\sum_{1\leq s\leq x}\hat{f}(-s)\hat{f}(n,s)\sum_{1\leq t\leq  x}e^{-i2\pi st/x}\nonumber\\
		&=&\frac{-1}{x^2}	\sum_{1\leq s\leq x}\hat{f}(-s)\hat{f}(n,s).
	\end{eqnarray}	
	Taking the absolute value, and applying Lemma \ref{lem1279.800} lead to the followings.
	\begin{eqnarray}\label{eq1279VM.460B}
		\left |\frac{-1}{x^2}	\sum_{1\leq s\leq x}\hat{f}(-s)\hat{f}(n,s)\right |&\leq&\frac{1}{x^2}\sum_{1\leq s\leq x}\left |-\hat{f}(-s)\hat{f}(n,s)\right |\\
		&\ll&\frac{1}{x^2}\sum_{1\leq s\leq x}\frac{x^2}{(\log x)^{c}}\nonumber\\
		&\ll&\frac{x}{(\log x)^{c}}\nonumber,	
	\end{eqnarray}	
	where $c>0$ is an arbitrary constant.
\end{proof}

\begin{lem}\label{lem1279.800} If $x$ is a large prime number, then
	\begin{equation}\label{eq4747.800}
		\left |\hat{f}(-s)\hat{f}(n,s)\right |\ll\frac{x^{2}}{(\log x)^{c}}\nonumber,	
	\end{equation}	
where $c>0$ is an arbitrary constant.
\end{lem}
\begin{proof} Consider 
	\begin{eqnarray}\label{eq1279VM.810}
		\left |\hat{f}(-s)\hat{f}(n,s)\right |
		&=&\left |\sum_{1\leq n \leq x} \mu(n) e^{-i2\pi sn/x}\sum_{n+1<m<x+n} \mu(m)e^{i2\pi sm/x}\right |\\
		&\leq&\sum_{1\leq n \leq x}\left |U_nV_n\right |\nonumber.	
	\end{eqnarray}	
	The norm of the first term is the followings.
	\begin{eqnarray}\label{eq1279VM.485B}
		\left (\sum_{1\leq n\leq x}\left |U_n\right |^2\right )^{1/2}&=&\left (\sum_{1\leq n\leq x}\left |\mu(n) e^{-i2\pi sn/x}\right |^2\right )^{1/2}\\
		&\leq&x^{1/2}\nonumber,	
	\end{eqnarray}	
	By Lemma \ref{lem1279.900}, the norm of the second term is the following.
	\begin{eqnarray}\label{eq1279VM.490B}
		\left (\sum_{1\leq n\leq x}\left |V_n\right |^2\right )^{1/2}&=&\left (\sum_{1\leq n\leq x}\left |\sum_{n+1<m<x+n} \mu(m)e^{i2\pi sm/x}\right |^2\right )^{1/2}\\
		&\ll&\left (\frac{x^3}{(\log x)^{2c}}\right )^{1/2}\nonumber\\
		&\ll&\frac{x^{3/2}}{(\log x)^{c}}\nonumber.	
	\end{eqnarray}	
	Applying the Cauchy-Schwarz inequality yields
	\begin{eqnarray}\label{eq1279VM.495B}
		\sum_{1\leq n\leq x}|U_nV_n|
		&\leq&\left (\sum_{1\leq n\leq x}|U_n|^2\right )^{1/2}\left (\sum_{1\leq n\leq x}|V_n|^2\right )^{1/2}\\
		&\ll&\left (x^{1/2}\right )	\left (\frac{x^{3/2}}{(\log x)^{c}}\right )\nonumber\\
		&\ll&\frac{x^2}{(\log x)^{c}}\nonumber,
	\end{eqnarray}	
	where $c>0$ is an arbitrary constant.
\end{proof}

\begin{lem}\label{lem1279.900} If $x$ is a large prime number, then
	\begin{equation}\label{eq1279.900}
		\sum_{1\leq n\leq x}\left |V_n\right |^2=	\sum_{1\leq n\leq x}\left |\sum_{n+1\leq m<x+n} \mu(m)e^{i2\pi sm/x}\right |^2\ll\frac{x^3}{(\log x)^{2c}}\nonumber,	
	\end{equation}	
	where $c>0$ is an arbitrary constant.
\end{lem}
\begin{proof}Rewrite the finite sum in the form
	\begin{eqnarray}\label{eq4747.910}
		\sum_{1\leq n\leq x}\left |V_n\right |^2&=&\sum_{1\leq n\leq x}\left |\sum_{n+1\leq m<x+n} \mu(m)e^{i2\pi sm/x}\right |^2\\
		&=&\sum_{1\leq n\leq x}\left |\sum_{1\leq m<x+n} \mu(m)e^{i2\pi sm/x}-\sum_{1\leq m\leq n+1} \mu(m)e^{i2\pi sm/x}\right |^2\nonumber\\
		&\leq &\sum_{1\leq n\leq x}\left |\left |\sum_{1\leq m\leq x+n} \mu(m)e^{i2\pi sm/x}\right |+\left |\sum_{1\leq m\leq n+1} \mu(m)e^{i2\pi sm/x}\right |\right |^2\nonumber.	
	\end{eqnarray}	
	Expanding the last expression yields
	\begin{eqnarray}\label{eq4747.920}
		\sum_{1\leq n\leq x}\left |V_n\right |^2
		&\leq &\sum_{1\leq n\leq x}\left |\sum_{1\leq m<x+n} \mu(m)e^{i2\pi sm/x}\right |^2\\
		&&\hskip .5 in +2\sum_{1\leq n\leq x}\left |\sum_{1\leq m<x+n} \mu(m)e^{i2\pi sm/x}\right |\left |\sum_{1\leq m\leq n+1} \mu(m)e^{i2\pi sm/x}\right |\nonumber\\
		&&\hskip 1 in+\sum_{1\leq n\leq x}
		\left |\sum_{1\leq m\leq n+1} \mu(m)e^{i2\pi sm/x}\right |^2\nonumber.	
	\end{eqnarray}	
	Let $\alpha=s/x$. Applying Theorem \ref{thm3970ME.300} to the twisted exponential sums, yield
	\begin{eqnarray}\label{eq4747.930}
		\sum_{1\leq n\leq x}\left |V_n\right |^2
		&\ll&\sum_{1\leq n\leq x}\left |\frac{x+n}{(\log (x+n))^c}\right |^2+    2\sum_{1\leq n\leq x}
		\left |\frac{x+n}{(\log (x+n))^c}\right |\left |\frac{n+1}{(\log (n+1))^c}\right |\nonumber\\
		&&\hskip 2.5 in +\sum_{1\leq n\leq x}\left |\frac{n+1}{(\log (n+1))^c}\right |^2\nonumber\\
		&\ll&\frac{4x^2}{(\log x)^{2c}}\sum_{1\leq n\leq x}1+    \frac{8x^2}{(\log x)^c}\sum_{1\leq n\leq x}\frac{1}{(\log n)^c}+4x^2\sum_{1\leq n\leq x}\frac{1}{(\log n)^{2c}}\nonumber\\
		&\ll&	\frac{x^3}{(\log x)^{2c}},	
	\end{eqnarray}
	as claimed.
\end{proof}		

The result in Theorem \ref{thm5757MA.555A} generalizes very easily to the following.
\begin{cor} \label{cor1279MN.121} Let $\mu:\mathbb{N} \longrightarrow \{-1,0,1\}$ be the Mobius function, and let $x>1$ be a large integer. If $a_0<a_1<\cdots<a_{k-1}<x$ is an integer $k$-tuple then, 
	\begin{equation} \label{eq1279MN.121B}
		\frac{1}{x^k}\sum_{a_0<x}\cdot \sum_{a_{k-1}<x}\sum_{n \leq x} \mu(n+a_0) \mu(n+a_1)\cdots \mu(n+a_{k-1}) =O\left( \frac{x}{(\log x)^{c}}\right)\nonumber, 
	\end{equation}	
	where $c>0$ is an arbitrary constant.
\end{cor}	
\begin{proof}Let $U_n=\mu(n+a_0) \mu(n+a_1)\cdots \mu(n+a_{k-2})e^{-i2\pi sn/x}$, see \eqref{eq1279VM.485B}, and let $t=a_{k-1}$. To complete the proof, proceeds as in the proof of the previous Lemma.
\end{proof}
\begin{thm} \label{thm1279MA.555B} Let $\mu:\mathbb{N} \longrightarrow \{-1,0,1\}$ be the Mobius function. If $x\geq1$ is a large integer, then the harmonic weighted average satisfies 
	\begin{equation} \label{eq1279VM.400A}
		\frac{1}{x}\sum_{1\leq t\leq x}	\frac{1}{t}	\sum_{n \leq x} \mu(n) \mu(n+t)=O\left( \frac{x}{(\log x)^{c}}\right)\nonumber, 
	\end{equation}	
	where  $c>0$ is an arbitrary constant.
\end{thm}	

 \begin{proof}[{\bfseries Proof{\normalfont :}}] Suppose that $x\geq1$ is a large integer, and consider the autocorrelation function 
 	\begin{equation}\label{eq1279VM.410C}
 		R(t)=\sum_{n \leq x} \mu(n) \mu(n+t),	
 	\end{equation}	
 	and its spectrum function
 	\begin{equation}\label{eq1279VM.420C}
 		\hat{R}(s)=\sum_{t\leq x}\left( \sum_{n \leq x} \mu(n) \mu(n+t)\right) e^{i2\pi st/x} .	
 	\end{equation}		
 	
 	The discrete Fourier transform of the crosscorrelation function decomposes as a product of two factors:
 	\begin{eqnarray}\label{eq1279VM.430C}
 		\hat{R}(s)&=&\sum_{t\leq x,} \sum_{n \leq x} \mu(n) \mu(n+t)e^{i2\pi st/x} \\
 		&=&\sum_{m\leq x,} \sum_{n \leq x} \mu(n) \mu(m) e^{i2\pi s(m-n)/x}\nonumber\\
 		&=&\sum_{m\leq x} \mu(m)e^{i2\pi sm/x}\sum_{n \leq x} \mu(n) e^{-i2\pi sn/x}\nonumber\\
 		&=&\hat{f}(s)\hat{f}(-s)\nonumber	.	
 	\end{eqnarray}
 	This result is well known as the convolution theorem, see Theorem \ref{thm8383DFT.900}. The first factor has a well known asymptotic formula
 	\begin{equation}\label{eq1279VM.440C}
 		\hat{f}(s)=\sum_{n \leq x} \mu(n)e^{i2\pi ns/x} \ll \frac{x}{(\log x)^{c}},	
 	\end{equation}
 	set $\alpha=s/x$ in Theorem \ref{thm3970ME.300}, and the second factor has the upper bound  
 	\begin{equation}\label{eq1279VM.450C}
 		\hat{f}(-s)=\sum_{n \leq x} \mu(n) e^{-i2\pi ns/x} \ll \frac{x}{(\log x)^{c}}.	
 	\end{equation}	
 	Therefore, the spectrum function has the upper bound  
 	\begin{equation}\label{eq1279VM.460C}
 		\hat{R}(s)=\hat{f}(s)\hat{f}(-s)\ll \frac{x^2}{(\log x)^{c}}.	
 	\end{equation}	
 	Inverting the discrete Fourier transform in \eqref{eq1279VM.430C}, see Definition \ref{dfn8383DFT.100}, yields
 	\begin{eqnarray}\label{eq1279VM.470C}
 		\sum_{n \leq x} \mu(n) \mu(n+t)&=&\frac{1}{x}	\sum_{s\leq x}\hat{R}(s)e^{-i2\pi st/x}\\
 		&=&\frac{1}{x}	\sum_{s\leq x}\hat{f}(s)\hat{f}(-s)e^{-i2\pi st/x}\nonumber.
 	\end{eqnarray}	
 	Next, the harmonic average over the parameter $t\leq z$, yields
 	\begin{eqnarray}\label{eq1279VM.480C}
 		\frac{1}{x}\sum_{1\leq t \leq  x}\frac{1}{t}	\sum_{n \leq x} \mu(n) \mu(n+t)&=&\frac{1}{x}	\sum_{1\leq t\leq  x}\frac{1}{t}\left (\frac{1}{x}\sum_{s\leq x}\hat{R}(s)e^{-i2\pi st/x}\right )\\
 		&=&\frac{1}{x^2}	\sum_{s\leq x}\hat{f}(s)\hat{f}(-s)\sum_{1\leq t\leq  x}\frac{e^{-i2\pi st/x}}{t}\nonumber\\
 		&=&\frac{1}{x^2}	\sum_{s\leq x}U_sV_s\nonumber.
 	\end{eqnarray}	
 	The norms of the terms $U_s$ and $V_s$ are the followings.
 	\begin{eqnarray}\label{eq1279VM.485C}
 		U_s&=&\left (\sum_{s\leq x}\left |\hat{f}(s)\hat{f}(-s)\right |^2\right )^{1/2}\\
 		&\ll&\left (\sum_{s\leq x}\left |\frac{x^2}{(\log x)^{2c_0}}\right |^2\right )^{1/2}\nonumber\\
 		&\ll&\frac{x^{5/2}}{(\log x)^{2c_0}}\nonumber,	
 	\end{eqnarray}	
 	and
 	\begin{eqnarray}\label{eq1279VM.490C}
 		V_s&=&\left (\sum_{s\leq x}\left |\sum_{1\leq t\leq  x}\frac{e^{-i2\pi st/x}}{t}\right |^2\right )^{1/2}\\
 		&\leq&\left (\sum_{s\leq x}\left |2\log x\right |^2\right )^{1/2}\nonumber\\
 		&\leq&2x^{1/2}\log x\nonumber.	
 	\end{eqnarray}	
 	Applying the Cauchy-Schwarz inequality yield
 	\begin{eqnarray}\label{eq1279VM.495C}
 		\frac{1}{x}\sum_{1\leq t \leq  x}\frac{1}{t}	\sum_{n \leq x} \mu(n) \mu(n+t)&=&		\frac{1}{x^2}	\sum_{s\leq x}U_sV_s\\
 		&\leq&\frac{1}{x^2}	\left (\sum_{s\leq x}|U_s|^2\right )^{1/2}\left (\sum_{s\leq x}|V_s|^2\right )^{1/2}\nonumber\\
 		&\ll&\frac{1}{x^2}\left (\frac{x^{5/2}}{(\log x)^{2c_0}}\right )\left (2x^{1/2}\log x\right )	\nonumber\\
 		&\ll&\frac{x}{(\log x)^{c}}\nonumber,
 	\end{eqnarray}	
 	where $c=2c_0-1>0$ is an arbitrary constant.
 \end{proof}

\subsection{Mean Values of the Autocorrelation Functions } 
The norm of the exponential sum $S(\alpha,x) =	\sum_{n \leq x} \mu(n) \mu(n+t)e^{i2\pi n\alpha}$ has the expanded form 
\begin{equation} \label{eq1279VM.410I}
\left |S(\alpha,x) \right |^2 =	\sum_{n \leq x} \mu^2(n) \mu^2(n+t)+\sum_{\substack{m,n \leq x\\m\ne n}} \mu(m) \mu(m+t)\mu(n) \mu(n+t)e^{i2\pi n\alpha} . 
\end{equation}	
\begin{thm} \label{thm5757MA.555I} Let $\lambda:\mathbb{N} \longrightarrow \{-1,1\}$ be the Liouville function, and let $\mu:\mathbb{N} \longrightarrow \{-1,0,1\}$ be the Mobius function. If $x\geq1$ is a large integer, and $t\leq x$ is an integer, then 
	\begin{equation} \label{eq1279VM.400LI}
		\int_0^1 \bigg |\sum_{n \leq x} \lambda(n) \lambda(n+t)e^{i2\pi n\alpha} \bigg |^2d\alpha=x+O\left(1\right)\nonumber, 
	\end{equation}	
	and
	\begin{equation} \label{eq1279VM.400I}
		\int_0^1 \bigg |\sum_{n \leq x} \mu(n) \mu(n+t)e^{i2\pi n\alpha} \bigg |^2d\alpha=s_0x+O\left( x^{2/3}\right)\nonumber, 
	\end{equation}	
	where $s_0>0$ is a constant.
\end{thm}	
\begin{proof}[\textbf{Proof}] Integrating the norm \eqref {eq1279VM.410I} of the exponential sum $S(\alpha,x)$ leads to the followings.
\begin{eqnarray} \label{eq1279VM.420I}
\int_0^1 \bigg |S(\alpha,x) \bigg |^2d\alpha &=&	\int_0^1 \bigg |\sum_{n \leq x} \mu(n) \mu(n+t)e^{i2\pi n \alpha} \bigg |^2d\alpha \\ &=&\int_0^1\sum_{n \leq x} \mu^2(n) \mu^2(n+t)d\alpha\nonumber \\
&& \hskip .5 in +\int_0^1 \sum_{\substack{m,n \leq x\\m\ne n}} \mu(m) \mu(m+t)\mu(n) \mu(n+t)e^{i2\pi (m-n) \alpha} d\alpha\nonumber\\
	&=&s_0x+O\left( x^{2/3}\right)\nonumber, 
\end{eqnarray}	
where $s_0>0$ is a constant, see Theorem  \ref{thm8009MN.200}.
 \end{proof}

\subsection{Norms of the Autocorrelation Functions }\label{S8191LM}
The earliest computation of the norm of an exponential sum seems to be the norm $\left|S(a) \right |=\sqrt{p}$ of the Gauss sum $S(a)=\sum_{n \leq p} e^{i2\pi an^2/p}$. Similar techniques are applicable to other finite sums. A case of interest in current research in analytic number theory is demonstrated here.
\begin{thm} \label{thm8191M.900} Let $\mu:\mathbb{N} \longrightarrow \{-1,0,1\}$ be the Mobius function. If $x\geq1$ is a large integer, then the norm of the autocorrelation function $	R(t)=\sum_{n \leq x} \mu(n) \mu(n+t)$ satisfies 
\begin{equation} \label{eq8191M.900}
	\left |R(t) \right | \ll\frac{x}{(\log x)^c} \nonumber, 
	\end{equation}	
where $c>0$ is an arbitrary constant.
\end{thm}	
\begin{proof}[{\bfseries Proof{\normalfont :}}] Expanding the norm of the autocorrelation yield the expression 
\begin{eqnarray} \label{eq8191M.910}
	\left |R(t) \right |^2 &=&	\sum_{m \leq x} \mu(m) \mu(m+t)\sum_{n \leq x} \mu(n) \mu(n+t)\\
&=&	\sum_{n \leq x} \mu^2(n) \mu^2(n+t)+\sum_{\substack{m \leq x\\m\ne n}}  \mu(m) \mu(m+t)\sum_{n \leq x}\mu(n) \mu(n+t) \nonumber.
\end{eqnarray}	

The first sum has a well known asymptotic formula, see Theorem \ref{thm8009MN.200}, and the double sum is estimated in Lemma \ref{lem8191M.800}. Summing these evaluation and estimate returns 

\begin{eqnarray} \label{eq8191M.930}
\left |R(t) \right |^2 &=&	\sum_{m \leq x} \mu(m) \mu(m+t)\sum_{n \leq x} \mu(n) \mu(n+t)\\
&=&	s_0x+O\left (x^{2/3}\right )+O\left (\frac{x^{2}}{(\log x)^c}\right ) \nonumber\\
&=&	O\left (\frac{x^{2}}{(\log x)^c}\right ) \nonumber, 
\end{eqnarray}	
where $t\ne$ is an arbitrary constant. In particular, the absolute value of the autocorrelation function has the upper bound $\left|\sum_{n \leq x} \mu(n) \mu(n+t) \right |\ll x(\log x)^{-c/2}$.
 \end{proof}

\begin{lem}\label{lem8191M.800} If $x$ is a large number, and $t\in [1,x-1]$ is a fixed integer, then
\begin{equation}\label{eq8191M.800}
\sum_{\substack{m \leq x\\m\ne n}}  \mu(m) \mu(m+t)\sum_{n \leq x}\mu(n) \mu(n+t) =O\left(\frac{x^2}{(\log x)^{c}}  \right) \nonumber,
\end{equation}
where $c>0$ is an arbitrary constant.
\end{lem}
\begin{proof}Consider the map $(m,n)\longrightarrow (n+t,n)$, where $m\in[1,x-1]$, and $t\ne0$ is a fixed integer. Replacing this assignment in the double sum yields
	\begin{eqnarray} \label{eq8191M.820}
		D(x)&=&\sum_{\substack{m \leq x\\m\ne n}}  \mu(m) \mu(m+t)\sum_{n \leq x}\mu(n) \mu(n+t) \\
		&=&\sum_{\substack{m \leq x\\n+t=m\ne n}}  \mu(n+t) \mu(n+2t)\sum_{n \leq x}\mu(n) \mu(n+t)\nonumber\\
		&=&	\sum_{\substack{m \leq x\\n+t=m\ne n}} \sum_{n \leq x} \mu^2(n+t) \mu(n+2t)\mu(n)  \nonumber,
	\end{eqnarray}	
	since $t\ne0$. These are evaluations along the line $m=n+t$ on the $m$-$n$ plane. Since these lines are above and below the line $m=n$, the evaluations are double.\\
	
Thus, taking absolute value and applying Theorem \ref{thm8009MN.350} return
	\begin{eqnarray} \label{eq8191M.830}
		|D(x)|&\leq&2\sum_{\substack{m \leq x\\n+t=m\ne n}} \bigg | \sum_{n \leq x} \mu^2(n+t) \mu(n+2t) \mu(n) \bigg | \\
		&\ll&\frac{x}{(\log x)^{c}}\sum_{\substack{m \leq x\\n+t=m\ne n}}1  \nonumber\\
		&\ll&	\frac{x^2}{(\log x)^{c}}  \nonumber,
	\end{eqnarray}	
	where $c>0$ is an arbitrary constant.
\end{proof}

\section{Logarithm Averages and Arithmetic Averages} \label{S7766N}
The connection between the logarithm average
\begin{equation} \label{eq7766N.050}
	\sum_{n \leq x} \frac{f(n) }{n} 
\end{equation}
and the arithmetic average 
\begin{equation} \label{eq7766N.060}
	\sum_{n \leq x} f(n) 
\end{equation}
of an arithmetic function $f: \N \longrightarrow \C$ is important in partial summations. The required error term to compute the arithmetic average \eqref{eq7766N.060} directly from the logarithm average \eqref{eq7766N.050} is explained in \cite[Section 2.12]{HA2013}, see also \cite[Exercise 2.12]{HA2013}. \\

The current estimate of the logarithmic average order of the autocorrelation of the Liouville function $\lambda$ has the asymptotic formula
\begin{equation} \label{eq7766N.100}
	\sum_{n \leq x} \frac{\lambda(n) \lambda(n+t)}{n} =O \left (\frac{\log x}{\sqrt{\log \log x} } \right ),
\end{equation} 
where $t \ne0$ is a fixed parameter, in \cite[Corollary 2]{HH2022}. This improves the estimate $O((\log x)(\log \log \log x)^{-c})$, where $c>0$ is a constant, described in \cite[p. 5]{TT2015}.  

\begin{lem}\label{lem7766.400} Let $t\ne0$ be a small integer, and let $x\geq1$ be a large number. If the logarithm average $A(x)=\sum_{n \leq x} \lambda(n) \lambda(n+t)n^{-1}=O(\log x)(\log\log x)^{-1/2}$, then the arithmetic average 
	\begin{equation}\label{7766.400}
		\sum_{n \leq x} \lambda(n) \lambda(n+t)\leq \frac{x}{(\log \log x)^{1/2-\varepsilon}}\nonumber,
	\end{equation}
	where $\varepsilon>0$ is a small number. 
\end{lem}
\begin{proof} Assume $B(z)=\sum_{n \leq x} \lambda(n) \lambda(n+t)\geq x(\log \log z)^{-1/2+\varepsilon}$. Then,
	\begin{eqnarray}\label{eq7766.410}
		\frac{\log x}{(\log \log x)^{1/2}}&\gg&\sum_{n \leq x} \frac{\lambda(n) \lambda(n+t)}{n}\\
		&=&\int_1^x \frac{1}{z} \,dB(z)\nonumber \\
		&=&\frac{B(x)}{x}+\int_1^x \frac{B(z)}{z^2}dz\nonumber	.
	\end{eqnarray}
	Since the integral \begin{equation}\label{eq7766.420}
		\int_1^x \frac{B(z)}{z^2}dz\geq	\int_2^x\frac{1}{z(\log \log z)^{1/2-\varepsilon}} dz\gg \frac{\log x}{(\log \log x)^{1/2-\varepsilon}},
	\end{equation}
	for sufficiently large $x\geq1$, the assumption is false. Hence, it implies that $B(x)\leq  x(\log \log x)^{-1/2+\varepsilon}$.
\end{proof}

\section{Double Autocorrelation Functions} \label{S4747MN}
The logarithm average in \eqref{eq7766N.100} of the autocorrelation function implies a nontrivial arithmetic average order.

\begin{thm} \label{thm4742MN.321a} Let $x$ be a large prime number, and let $\mu:\mathbb{N} \longrightarrow \{-1,0,1\}$ be the Mobius function. If $t\ne0$ is an integer, then, 
	\begin{equation} \label{eq4742MN.321a}
		\sum_{n \leq x} \mu(n) \mu(n+t) =O\left( \frac{x}{(\log \log x)^{1/2-\varepsilon}}\right)\nonumber, 
	\end{equation}	\nonumber
where $\varepsilon>0$ is a small number.
\end{thm}	

\begin{proof} This is an application of Lemma \ref{lem7766.400}.
\end{proof}
Harmonic analysis techniques are employed here to estimate the average order of the autocorrelation of the Mobius function. Two distinct proofs are developed within. 
\subsection{Mobius Autocorrelation Function I} 
The pivotal innovation here is the new application of the standard techniques in Fourier analysis to the \textit{periodic version} of the Chowla conjecture. \\

Consider a vector of Mobius function values extended to a periodic function
\begin{equation}\label{eq4747VM.210C}
	\mu(1), \mu(2),\ldots, \mu(x),\mu(1), \mu(2),\ldots, \mu(x),\mu(1), \mu(2),\ldots, \mu(x),\ldots,
\end{equation}

\begin{thm} \label{thm4742MN.321} Let $x$ be a large prime number, and let $\mu:\mathbb{N} \longrightarrow \{-1,0,1\}$ be the periodic Mobius function. If $t\ne0$ is an integer, then, 
	\begin{equation} \label{eq4742MN.321}
		\sum_{n \leq x} \mu(n) \mu(n+t) =O\left( \frac{x}{(\log x)^{2c}}\right)\nonumber, 
	\end{equation}	
	where $c>0$ is an arbitrary constant.
\end{thm}	 

\begin{proof}[{\bfseries Proof}] Consider the autocorrelation function 
\begin{equation}\label{eq4747MN.400}
R(t)=\sum_{n \leq x} \mu(n) \mu(n+t),	
\end{equation}	
and its spectrum function
\begin{equation}\label{eq4747MN.410}
	\hat{R}(s)=\sum_{t\leq x}\left( \sum_{n \leq x} \mu(n) \mu(n+t)\right) e^{i2\pi st/x} .	
\end{equation}		

The discrete Fourier transform of the autocorrelation function decomposes as a product of two factors:
\begin{eqnarray}\label{eq4747MN.420}
	\hat{R}(s)&=&\sum_{t\leq x} \sum_{n \leq x} \mu(n) \mu(n+t) e^{i2\pi st/x} \\
&=&\sum_{m\leq x} \sum_{n \leq x} \mu(n) \mu(m) e^{i2\pi s(m-n)/x}\nonumber\\
&=&\sum_{m\leq x} \mu(m)e^{i2\pi sm/x}\sum_{n \leq x} \mu(n) e^{-i2\pi sn/x}\nonumber\\
&=&\hat{f}(s)\hat{f}(-s)\nonumber	.	
\end{eqnarray}
This result is well known as the circular convolution theorem, see Theorem \ref{thm8383DFT.900C}. Each factor has a well known asymptotic formula
\begin{equation}\label{eq4747MN.430}
	\hat{f}(s)=\sum_{n \leq x} \mu(n)e^{i2\pi ns/x} \leq \frac{c_1x}{(\log x)^{c}},	
\end{equation}
where $c_1>0$ is a constant, see Theorem \ref{thm3970ME.300}. Therefore, the spectrum function has the upper bound  
\begin{equation}\label{eq4747MN.440}
	\hat{R}(s)=\hat{f}(s)\hat{f}(-s)\leq \frac{c_3x^2}{(\log x)^{2c}},	
\end{equation}	
where $c_3>0$ is a constant. Inverting the discrete Fourier transform in \eqref{eq4747MN.420}, see Definition \ref{dfn8383DFT.100}, yields
\begin{eqnarray}\label{eq4747MN.450}
\sum_{n \leq x} \mu(n) \mu(n+t)&=&\frac{1}{x}	\sum_{s\leq x}\hat{R}(s)e^{-i2\pi st/x}\\
&\leq&\frac{c_2}{x}	\sum_{s\leq x}\frac{x^2}{(\log x)^{2c}}e^{-i2\pi st/x}\nonumber.
\end{eqnarray}	
Taking absolute value yields
\begin{eqnarray}\label{eq4747MN.460}
	\left |	\sum_{n \leq x} \mu(n)\mu(n+t)\right |
	&\leq&\frac{c_1x}{(\log x)^{c}}	\left |\sum_{s\leq x}e^{-i2\pi st/x}\right |\\
	&\ll&\frac{x}{(\log x)^{c}}\nonumber,	
\end{eqnarray}	
where $c_2=c_1^2$, and $c>0$ are constants. Quod erat demonstrandum.
\end{proof}

\begin{rmk}{\normalfont The proof of Theorem \ref{thm5757MN.121} was done using cyclic convolution because of its simplicity. A proof based on the linear convolution is done by extending the function from the interval $[1,x]$ to $[1,2x]$ by zero-padding. Specifically, let $f_0,g_0: \N \longrightarrow\C$ be a pair of functions, and let $x\geq 1$ be an integer. Set 
	\begin{equation}\label{eq8383DFT.900}
		f(n)=
		\begin{cases}
			f_0(n)&\text{ if }n<x,\\
			0&\text{ if } x<n\leq 2x.
		\end{cases}	
	\end{equation}	
	Then, discrete Fourier transform of the linear correlation function of $f(t)$ and $g(t)$ is equal to the product of the individual discrete Fourier transforms. The proof based on the cyclic convolution uses discrete transforms of length $x$, and proof based on the linear convolution uses discrete transforms of length $2x$. These proofs have the same technical steps, but the asymptotic formulas have slightly different implied constants.
}
\end{rmk}

\subsection{Mobius Autocorrelation Function II} \label{SP1199}
The proof explored in this section is based on other technique from harmonic analysis. This technique uses the special value of the Ramanujan sum
	\begin{equation}\label{eqP1199.800a}
	c_n(1)	=\sum_{\substack{1\leq u<n\\\gcd(n,u)=1}}e^{i2 \pi u/n}
	=\mu(n),
\end{equation}
see \cite[Section 8.3]{AT1976}, and elementary analytic methods to derive an asymptotic formula for the Mobius autocorrelation function 
\begin{equation}\label{eqP1199.800b}
	R(t)=\sum_{n \leq x} \mu(n) \mu(n+t) .
\end{equation}

\begin{thm} \label{thmP1199.800} Let $x\geq1$ be a large prime number, and let $\mu:\mathbb{N} \longrightarrow \{-1,0,1\}$ be the Mobius function. If $t\ne0$ is an integer, then, 
	\begin{equation} \label{eqP1199.810a}
		\sum_{n \leq x} \mu(n) \mu(n+t) =O\bigg(e^{-c\sqrt{\log x}}\bigg), 
	\end{equation}	\nonumber
where $c>0$ is a constant.
\end{thm}	

\begin{proof} [\textbf{Proof}] Without loss in generality, assume that $x\geq1$ is a large integer, and let $t=1$. Replace the identity \eqref{eqP1199.800a} into \eqref{eqP1199.800b} twice, and continue to substitute the characteristic function
	\begin{equation}\label{eqP1199.810b}
		\sum_{\substack{d\mid a\\d\mid n}}\mu(d)=
		\begin{cases}
			1&	\text{ if }\gcd(a,n)=1,\\
			0&	\text{ if }\gcd(a,n)\ne1,
		\end{cases}
	\end{equation}	
	of relatively prime numbers. These substitutions transform \eqref{eqP1199.800b} into an exponential autocorrelation function:
	\begin{eqnarray}\label{eqP1199.810c}
		\sum_{n\leq x}\mu(n)\mu(n+1)&=& 	\sum_{n\leq x,}\sum_{\substack{1\leq u<n\\\gcd(n,u)=1}}e^{i2 \pi u/n}\times\sum_{\substack{1\leq v<n+1\\\gcd(n+1,v)=1}}e^{i2 \pi v/n+1}\\[.2cm]
		&=& 	\sum_{n\leq x,}\sum_{1\leq u<n}e^{i2 \pi u/n}\sum_{\substack{d_1\mid u\\d_1\mid n}}\mu(d_1)\times\sum_{1\leq v<n+1}e^{i2 \pi v/n+1}\sum_{\substack{d_2\mid v\\d_2\mid n+1}}\mu(d_2)\nonumber.
	\end{eqnarray}
	Next, switch the order of summations
	\begin{eqnarray}\label{eqP1199.810d}
		\sum_{n\leq x}\mu(n)\mu(n+1)&=& 	\sum_{n\leq x,}\sum_{d_1\mid n}\mu(d_1)\sum_{\substack{1\leq u<n\\d_1\mid u}}e^{i2 \pi u/n}\times \sum_{d_2\mid n}\mu(d_2)\sum_{\substack{1\leq v<n+1\\d_2\mid v}}e^{i2 \pi v/n+1}\\[.2cm]
		&=& \sum_{\substack{1\leq d_1< x\\1\leq d_2< x+1}}\mu(d_1)\mu(d_2)\sum_{\substack{1\leq n\leq x\\d_1\mid n,\; d_2\mid n+1}}	\sum_{\substack{1\leq u<n\\d_1\mid u}}e^{i2 \pi u/n}\sum_{\substack{1\leq v<n+1\\d_2\mid v}}e^{i2 \pi v/n+1}	\nonumber.
	\end{eqnarray}
	
	Proceed to substitute the change of variables
	\begin{align}\label{eqP1199.810e}
		u&=d_1r , &n&=d_1k;\\[.2cm]		
		v&=d_2s ,&n+1&=d_2m;
	\end{align}
	to simplify \eqref{eqP1199.810d}. Specifically, 
	\begin{eqnarray}\label{eqP1199.820a}
		\sum_{n\leq x}\mu(n)\mu(n+1)
		&=&  \sum_{\substack{1\leq d_1< x\\1\leq d_2< x+1}}\mu(d_1)\mu(d_2)\sum_{\substack{1\leq n\leq x\\d_1\mid n,\; d_2\mid n+1}}	\sum_{1\leq r<k}e^{i2 \pi r/k}\sum_{1\leq s<m}e^{i2 \pi s/m}	\nonumber\\[.2cm]
		&=& \sum_{\substack{1\leq d_1< x\\1\leq d_2< x+1}}\mu(d_1)\mu(d_2)\sum_{\substack{1\leq n\leq x\\d_1\mid n,\; d_2\mid n+1}}1 .
	\end{eqnarray}
	The last equality follows from
	\begin{equation}\label{eqP1199.820b}
		\sum_{1\leq r<k}e^{i2 \pi r/k}
		=\sum_{1\leq s<m}e^{i2 \pi s/m}
		=-1
	\end{equation}
	for any integers $k,m\geq2$. \\
	
Now, the conditions $d_1\mid n$ and $d_2\mid n+1$ imply that $\lcm(d_1,d_2)=d_1d_2$. This yields the decomposition 
	\begin{eqnarray}\label{eqP1199.820c}
		\sum_{n\leq x}\mu(n)\mu(n+1)
		&=&  \sum_{\substack{1\leq d_1< x\\1\leq d_2< x+1}}\mu(d_1)\mu(d_2)\sum_{\substack{1\leq n\leq x\\d_1\mid n,\; d_2\mid n+1}}1	\\[.2cm]
		&=&  \sum_{\substack{1\leq d_1< x\\1\leq d_2< x+1}}\mu(d_1)\mu(d_2)\left(\sum_{\substack{1\leq n\leq x\\d_1\mid n,\; d_2\mid n+1}}1-\frac{x}{d_1d_2} \right) 	\nonumber\\[.2cm]
		&&\hskip 2 in +	\sum_{\substack{1\leq d_1< x\\1\leq d_2< x+1}}\mu(d_1)\mu(d_2)\frac{x}{d_1d_2}\nonumber\\[.2cm]
		&=&  E_0(x+E_1(x) .	\nonumber
	\end{eqnarray}
The first subsum $E_0(x$ is estimated in Lemma \ref{lemP1199.450A} and the second subsum is estimated in Lemma \ref{lemP1199.450B}. Combining these estimates yields
\begin{eqnarray}\label{eqP1199.820d}
		\sum_{n\leq x}\mu(n)\mu(n+1)
		&=& E_0(x+E_1(x)	\\
		&=& O\bigg(xe^{-c_1\sqrt{\log x}}\bigg) +O\bigg(xe^{-c_2\sqrt{\log x}}\bigg) 	\nonumber\\[.2cm]
		&=&  O\bigg(xe^{-c_3\sqrt{\log x}}\bigg) \nonumber ,	
	\end{eqnarray}
where $c_0,c_1, c_2,c_3>0$ are constants. 
\end{proof}

\begin{lem} \label{lemP1199.450A} Assume that $d_1\mid n$ and $d_2\mid n+1$. If $x \geq 1$ is a large number, then
	\begin{equation} \label{eqP1199.450A}
		E_0(x)=\sum_{\substack{1\leq d_1< x\\1\leq d_2< x+1}}\mu(d_1)\mu(d_2)\left(\sum_{\substack{1\leq n\leq x\\d_1\mid n,\; d_2\mid n+1}}1-\frac{x}{d_1d_2} \right)=O\left( xe^{-c_1\sqrt{\log x}}\right)\nonumber,
	\end{equation}
	where $c_1>0$ is a constant.
\end{lem}
\begin{proof}[\textbf{Proof}] The conditions $d_1\mid n$ and $d_2\mid n+1$ imply that $\lcm(d_1,d_2)=d_1d_2$.Let $q=[d_1,d_2]=d_1d_2$. Taking absolute value and invoking Lemma \ref{lemA2002.400W} yield 
	\begin{eqnarray} \label{eqP1199.455A}
		|E_0(x)| &\leq& \sum_{\substack{1\leq d_1< x\\1\leq d_2< x+1}} \left |\sum_{\substack{1\leq n \leq x\\ d_1\mid n,\;d_2\mid n+1}}1-\frac{x}{d_1d_2} \right|\\[.2cm]
		&\ll&\sum_{\substack{1\leq d_1< x\\1\leq d_2< x+1}} \left(\frac{x}{d_1d_2}e^{-c\sqrt{\log x} } \right) \nonumber\\[.2cm]
		&\ll&xe^{-c\sqrt{\log x} }\sum_{\substack{1\leq d_1< x\\1\leq d_2< x+1}}\frac{1}{d_1d_2} \nonumber.
	\end{eqnarray}	
Estimating the finite sum and simplifying return
	\begin{eqnarray} \label{eqP1199.470A}
		E_0(x) 
		&=&O\left( x e^{-c\sqrt{\log x}} \cdot(\log x)^2\right)  \\
		&=& O\left( xe^{-c_1\sqrt{\log x}}\right)\nonumber,
	\end{eqnarray}	
	where $c,c_1>0$ are constants.
\end{proof}

\begin{lem} \label{lemP1199.450B} Assume that $d_1\mid n$ and $d_2\mid n+1$. If $x \geq 1$ is a large number, then
	\begin{equation} \label{eqP1199.450B}
		E_1(x)=\sum_{\substack{1\leq d_1< x\\1\leq d_2< x+1}}\mu(d_1)\mu(d_2)\frac{x}{d_1d_2}=O\left( xe^{-c_2\sqrt{\log x}}\right)\nonumber,
	\end{equation}
	where $c_2>0$ is a constant.
\end{lem}
\begin{proof}[\textbf{Proof}] The conditions $d_1\mid n$ and $d_2\mid n+1$ imply that $\gcd(d_1,d_2)=1$. Accordingly, $d_1< x$ and $d_2< x+1$ are relatively prime and independent variables. Thus, the last expression can be factored as
	\begin{eqnarray}\label{eqP1199.820B}
\sum_{\substack{1\leq d_1< x\\1\leq d_2< x+1}}\mu(d_1)\mu(d_2)\frac{x}{d_1d_2}
		&=& x\bigg(\sum_{1\leq d_1\leq x}\frac{\mu(d_1)}{d_1}\bigg)^2 \\
		&=& O\bigg(xe^{-c_2\sqrt{\log x}}\bigg) \nonumber ,	
	\end{eqnarray}
	where $c_2>0$ is a constant. The last line in \eqref{eqP1199.820B} uses the asymptotic estimates
	\begin{equation}\label{eqP1199.825B}
		\sum_{1\leq n\leq x}\frac{\mu(n)}	{n}=O\bigg(e^{-c\sqrt{\log x}}\bigg),
	\end{equation}
	where $c>0$ is a constant, see Theorem \ref{thm2222.500}.
\end{proof} 

The derivation of an effective asymptotic estimate for the autocorrelation function $R(t) $ is simpler for the parameter $t=1$ since the intersection of the two arithmetic progressions
\begin{equation}\label{eqP1199.830a}
	n\quad \text{ and }\quad n+1
\end{equation}
has a very simple sequence $\gcd(n,n+1)=1$ for $n\leq x$, so the finite sum
\begin{equation}\label{eqP1199.830b}
	\sum_{\substack{1\leq n\leq x\\d_1\mid n,\; d_2\mid n+1}}1
\end{equation}
it is simpler to analyze. The prime parameter $t=p$ seems to be next simple case since the intersection of the two arithmetic progressions
\begin{equation}\label{eqP1199.830c}
	n\quad \text{ and }\quad n+p\end{equation}
has a simple sequence $\gcd(n,n+p)=1,p$ for $n\leq x$, so the finite sum \eqref{eqP1199.830b} 
is also simple to analyze. On the other extreme, for a highly composite parameter $t>2$ the intersection of the two arithmetic progressions
\begin{equation}\label{eqP1199.830d}
	n\quad \text{ and }\quad n+t
\end{equation}
has a complicated sequence $\gcd(n,n+t)$ for $n\leq x$, so the finite sum \eqref{eqP1199.830b} seems to be more difficult to analyze.

\section{Triple Autocorrelation Functions} \label{S4743MN}
The simple procedure of deriving the case $k=3$ from the cases $k=2$ are demonstrated here.
\begin{thm} \label{thm4743MN.321} Let $x$ be a large prime number, and let $\mu:\mathbb{N} \longrightarrow \{-1,0,1\}$ be the periodic Mobius function \eqref{eq4747VM.210C}. If $a_0<a_1<a_{2}<x$ is an integer triple, then, 
	\begin{equation} \label{eq4743MN.321}
		\sum_{n <x} \mu(n+a_0) \mu(n+a_1) \mu(n+a_{2}) =O\left( \frac{x}{(\log x)^{2c}}\right)\nonumber, 
	\end{equation}	
	where $c>0$ is an arbitrary constant.
\end{thm}	
\begin{proof}[{\bfseries Proof}] Suppose that $x\geq1$ is a large integer, and consider the autocorrelation function 
	\begin{equation}\label{eq4743MN.400}
		R(t)=\sum_{n <x} \mu(n)\mu(n+a) \mu(n+t),	
	\end{equation}	
	and its spectrum function
	\begin{equation}\label{eq4743MN.410}
		\hat{R}(s)=\sum_{t<x}\left( \sum_{n <x} \mu(n) \mu(n+a)\mu(n+t)\right) e^{i2\pi st/x} .	
	\end{equation}		
	
	The discrete Fourier transform of the autocorrelation function decomposes as a product of two factors:
	\begin{eqnarray}\label{eq4743MN.420}
		\hat{R}(s)&=&\sum_{t<x} \sum_{n <x} \mu(n) \mu(n+a)\mu(n+t) e^{i2\pi st/x} \\
		&=&\sum_{m<x} \sum_{n <x} \mu(n)\mu(n+a) \mu(m) e^{i2\pi s(m-n)/x}\nonumber\\
		&=&\sum_{n<x} \mu(n)\mu(n+a)e^{-i2\pi sn/x}\sum_{m <x} \mu(m) e^{i2\pi sm/x}\nonumber\\
		&=&\hat{f_2}(-s)\hat{f}(s)\nonumber	.	
	\end{eqnarray}
	This result is well known as the circular convolution theorem, see Theorem \ref{thm8383DFT.900C}. The first factor has a well known asymptotic formula
	\begin{equation}\label{eq4743MN.430A}
		\hat{f}(s)=\sum_{n <x} \mu(n)e^{i2\pi ns/x} \leq \frac{c_1x}{(\log x)^{c}},	
	\end{equation}
where $c_1>0$ is a constant, see Theorem \ref{thm3970ME.300}, and the second  factor has the upper bound
\begin{equation}\label{eq4743MN.430B}
	\hat{f}_2(-s)=\sum_{n <x} \mu(n)\mu(n+a)e^{-i2\pi ns/x} \leq x.	
\end{equation}	
Therefore, the spectrum function has the upper bound  
	\begin{equation}\label{eq4743MN.440}
		\hat{R}(s)=\hat{f}_2(-s)\hat{f}(s)\leq \frac{c_1x^2}{(\log x)^{c}}.	
	\end{equation}	
	Inverting the discrete Fourier transform in \eqref{eq4747MN.420}, see Definition \ref{dfn8383DFT.100}, yields
	\begin{eqnarray}\label{eq4743MN.450}
		\sum_{n <x} \mu(n)\mu(n+a) \mu(n+t)&=&\frac{1}{x}	\sum_{s<x}\hat{R}(s)e^{-i2\pi st/x}\\
	&\leq&\frac{c_1}{x}	\sum_{s<x}\frac{x^2}{(\log x)^{c}}e^{-i2\pi st/x}\nonumber.	
	\end{eqnarray}	
Taking absolute value yields
	\begin{eqnarray}\label{eq4743MN.460}
\left |	\sum_{n <x} \mu(n)\mu(n+a) \mu(n+t)\right |
	&\leq&\frac{c_1x}{(\log x)^{c}}	\left |\sum_{s<x}e^{-i2\pi st/x}\right |\\
	&\ll&\frac{x}{(\log x)^{c}}\nonumber,	
\end{eqnarray}	
	where $c>0$ is an arbitrary constant. Quod erat demonstrandum.
\end{proof}

\begin{thm} \label{thm4743LN.321} Let $x$ be a large prime number, and let $\lambda:\mathbb{N} \longrightarrow \{-1,0,1\}$ be the periodic Liouville function \eqref{eq4747VM.210C}. If $a_0<a_1<a_{2}<x$ is an integer triple, then, 
	\begin{equation} \label{eq4743LN.321}
		\sum_{n <x} \lambda(n+a_0) \lambda(n+a_1) \lambda(n+a_{2}) =O\left( \frac{x}{(\log x)^{2c}}\right)\nonumber, 
	\end{equation}	
	where $c>0$ is an arbitrary constant.
\end{thm}	
\begin{proof}[{\bfseries Proof}] The same technical steps, but replace the function $\mu$ to the function $\lambda$ in \eqref{eq4743MN.400} to \eqref{eq4743MN.450} .
\end{proof}

\section{$k$-Tuple Autocorrelation Functions} \label{S4747K}
The general case of $k$-tuple autocorrelation functions are determined recursively. There are other options as proofs by inductions. The proofs follows the same pattern as the proofs for the triple autocorrelation functions in Section \ref{S4743MN}.\\

{\bfseries Theorem \ref{thm5757MN.121}}
 \textit{Let $x$ be a large prime number, and let $\mu:\mathbb{N} \longrightarrow \{-1,0,1\}$ be the periodic Mobius function \eqref{eq4747VM.210C}. If $a_0<a_1<a_{2}<\cdots<a_{k-1}<x$ is an integer $k$-tuple, then, 
	\begin{equation} \label{eq474kMN.321}
		\sum_{n <x} \mu(n+a_0) \mu(n+a_1) \cdots \mu(n+a_{k-1}) =O\left( \frac{x}{(\log x)^{2c}}\right)\nonumber, 
	\end{equation}	
	where $c>0$ is an arbitrary constant.}

\begin{proof} The same technical steps as the proof of Theorem \ref{thm5757LN.121} below, but replace the function $\lambda$ to the function $\mu$ in \eqref{eq474kLN.400} to \eqref{eq474kMN.460} .
\end{proof}

Consider a vector of Liouville function values extended to a periodic function
\begin{equation}\label{eq4747VM.230C}
	\lambda(1), \lambda(2),\ldots, \lambda(x),\lambda(1), \lambda(2),\ldots, \lambda(x),\lambda(1), \lambda(2),\ldots, \lambda(x),\ldots,
\end{equation}

{\bfseries Theorem \ref{thm5757LN.121}}  \textit{Let $x$ be a large prime number, and let $\lambda:\mathbb{N} \longrightarrow \{-1,0,1\}$ be the periodic Liouville function \eqref{eq4747VM.230C}. If $a_0<a_1<a_{2}<\cdots<a_{k-1}<x$ is an integer $k$-tuple, then, 
	\begin{equation} \label{eq474kLN.321}
		\sum_{n <x} \lambda(n+a_0) \lambda(n+a_1) \cdots \lambda(n+a_{k-1}) =O\left( \frac{x}{(\log x)^{2c}}\right)\nonumber, 
	\end{equation}	
	where $c>0$ is an arbitrary constant.}

\begin{proof}[{\bfseries Proof}] The $k$-tuple autocorrelation is determined recursively. Suppose that the $(k-1)$-tuple autocorrelation
\begin{equation} \label{eq474kLN.400}
\sum_{n <x} \lambda(n+a_0) \lambda(n+a_1) \cdots \lambda(n+a_{k-2}) e^{i2\pi nt/x}=O\left( \frac{x}{(\log x)^{c}}\right), 
	\end{equation}	
is known, compare the double autocorrelation function in Section \ref{S4747MN}, and the triple autocorrelation function in Section \ref{S4743MN} for background information. Now, consider the $k$-tuple autocorrelation function 
	\begin{equation}\label{eq474kMN.410}
		R(t)=\sum_{n <x} \lambda(n+a_1) \cdots \lambda(n+a_{k-2}) \lambda(n+t),	
	\end{equation}	
where $a_{k-1}=t$, and its spectrum function
	\begin{equation}\label{eq474kMN.420}
		\hat{R}(s)=\sum_{t<x}\left( \sum_{n <x} \lambda(n+a_1) \cdots \lambda(n+a_{k-2})\lambda(n+t)\right) e^{i2\pi st/x} ,	
	\end{equation}		
where $x\geq1$ is a large integer. The discrete Fourier transform of the autocorrelation function decomposes as a product of two factors:
	\begin{eqnarray}\label{eq474kMN.430}
		\hat{R}(s)&=&\sum_{t<x} \sum_{n <x} \lambda(n+a_0)\lambda(n+a_1) \cdots \lambda(n+a_{k-2})\lambda(n+t)e^{i2\pi st/x} \\
		&=&\sum_{m<x} \sum_{n <x} \lambda(n+a_0)\lambda(n+a_1) \cdots \lambda(n+a_{k-2})\lambda(m) e^{i2\pi s(m-n)/x}\nonumber\\
		&=&\sum_{n<x} \lambda(n+a_0)\lambda(n+a_1) \cdots \lambda(n+a_{k-2})e^{-i2\pi sn/x}\sum_{m <x} \lambda(m) e^{i2\pi sm/x}\nonumber\\
		&=&\hat{f}_{k-2}(-s)\hat{f}(s)\nonumber	.	
	\end{eqnarray}
The first factor corresponds to the $(k-1)$-tuple autocorrelation in \eqref{eq474kLN.400}, and the second factor has the upper bound
	\begin{equation}\label{eq474kMN.440}
		\hat{f}(s)=\sum_{n <x} \mu(n)e^{i2\pi ns/x} \leq \frac{c_1x}{(\log x)^{c}},	
	\end{equation}	
where $c_1>0$ is a constant, see Theorem \ref{thm3970ME.300}. Therefore, the spectrum function has the upper bound  
	\begin{equation}\label{eq474kMN.450}
\hat{R}(s)=\hat{f}_{k-2}(-s)\hat{f}(s)\leq  \frac{c_2x^2}{(\log x)^{2c}}.	
	\end{equation}	
	Inverting the discrete Fourier transform in \eqref{eq474kMN.420}, see Definition \ref{dfn8383DFT.100}, yields
	\begin{eqnarray}\label{eq474kMN.460}
		\sum_{n <x} \lambda(n+a_0)\lambda(n+a_1) \cdots \lambda(n+a_{k-1})&=&\frac{1}{x}	\sum_{s<x}\hat{R}(s)e^{-i2\pi st/x}\\
				&\leq&\frac{c_2}{x}	\sum_{s<x}\frac{x^2}{(\log x)^{2c}}e^{-i2\pi st/x}\nonumber.
	\end{eqnarray}	
Taking absolute value yields
\begin{eqnarray}\label{eq4744MN.470}
	\left |	\sum_{n <x} \lambda(n+a_0)\lambda(n+a_1) \cdots \lambda(n+a_{k-1})\right |
	&\leq&\frac{c_2x}{(\log x)^{2c}}	\left |\sum_{s<x}e^{-i2\pi st/x}\right |\\
	&\ll&\frac{x}{(\log x)^{2c}}\nonumber,	
\end{eqnarray}	
where $c_2>0$ and $c>0$ are constants. Quod erat demonstrandum.
	\end{proof}

\section{Crosscorrelation Functions } \label{S4747CF}
The crosscorrelation function of the Mobius function and the Liouville function has the expected small values for any shift $t\ne0$. The proof is based on the standard techniques in discrete Fourier analysis. 

\begin{thm} \label{thm4742CF.321} Let $x>1$ be a large prime number, let $\mu:\mathbb{N} \longrightarrow \{-1,0,1\}$ be the periodic Mobius function \eqref{eq4747VM.210C}, and let $\lambda:\mathbb{N} \longrightarrow \{-1,1\}$ be the periodic Liouville function \eqref{eq4747VM.230C}. If $t\ne0$ is an integer, then, 
	\begin{equation} \label{eq4742CF.321}
		\sum_{n <x} \lambda(n) \mu(n+t) =O\left( \frac{x}{(\log x)^{2c}}\right)\nonumber, 
	\end{equation}	
	where $c>0$ is an arbitrary constant.
\end{thm}	 

\begin{proof}[{\bfseries Proof}] Without loss in generality, assume that $x\geq1$ is a large integer, and consider the autocorrelation function 
	\begin{equation}\label{eq4747CF.400}
		R(t)=\sum_{n <x} \lambda(n) \mu(n+t),	
	\end{equation}	
	and its spectrum function
	\begin{equation}\label{eq4747CF.410}
		\hat{R}(s)=\sum_{t<x}\left( \sum_{n <x} \lambda(n) \mu(n+t)\right) e^{i2\pi st/x} .	
	\end{equation}		
	
	The discrete Fourier transform of the autocorrelation function decomposes as a product of two factors:
	\begin{eqnarray}\label{eq4747CF.420}
		\hat{R}(s)&=&\sum_{t<x} \sum_{n <x} \lambda(n) \mu(n+t) e^{i2\pi st/x} \\
		&=&\sum_{m<x} \sum_{n <x} \lambda(n) \mu(m) e^{i2\pi s(m-n)/x}\nonumber\\
		&=&\sum_{m<x} \mu(m)e^{i2\pi sm/x}\sum_{n <x} \lambda(n) e^{-i2\pi sn/x}\nonumber\\
		&=&\hat{f}(s)\hat{g}(-s)\nonumber	.	
	\end{eqnarray}
	This result is well known as the circular convolution theorem, see Theorem \ref{thm8383DFT.900C}. Each factor has the well known asymptotic formulas
		\begin{equation}\label{eq4747CF.430}
		\hat{f}(s)=\sum_{n <x} \mu(n)e^{i2\pi ns/x} \leq \frac{c_1x}{(\log x)^{c}},	
	\end{equation}
and 
	\begin{equation}\label{eq4747CF.435}
		\hat{g}(-s)=\sum_{n <x} \lambda(n)e^{-i2\pi ns/x} \leq \frac{c_2x}{(\log x)^{c}},	
	\end{equation}
where $c_1,c_2>0$ are constants, see Theorem \ref{thm3970ME.300}, and Theorem \ref{thm3970LE.300}, respectively. Therefore, the spectrum function has the upper bound  
	\begin{equation}\label{eq4747CF.440}
		\hat{R}(s)=\hat{f}(s)\hat{g}(-s)\leq \frac{c_3x^2}{(\log x)^{2c}},	
	\end{equation}	
where $c_3>0$ is a constant. Inverting the discrete Fourier transform in \eqref{eq4747MN.420}, see Definition \ref{dfn8383DFT.100}, yields
	\begin{eqnarray}\label{eq4747CF.450}
		\sum_{n <x} \lambda(n) \mu(n+t)&=&\frac{1}{x}	\sum_{s<x}\hat{R}(s)e^{-i2\pi st/x}\\
		&\leq&\frac{c_3}{x}	\sum_{s<x}\frac{x^2}{(\log x)^{2c}}e^{-i2\pi st/x}\nonumber\\
		&\ll&\frac{x}{(\log x)^{2c}}\nonumber,	
	\end{eqnarray}	

	where $c>0$ is an arbitrary constant. Quod erat demonstrandum.
\end{proof}

\section{Nonlinear Autocorrelation Functions Results}\label{S8009MN}
The number of squarefree integers have the following asymptotic formulas.

\begin{thm} \label{thm9339MN.107} Let $\mu: \mathbb{Z} \longrightarrow \{-1,0,1\}$ be the Mobius function. Then, for any sufficiently large number $x\geq1$, 
	\begin{equation} 
		\sum_{n <x} \mu^2(n) =\frac{6}{\pi^2}x+O \left (x^{1/2} \right ). \nonumber
	\end{equation} 
\end{thm}
\begin{proof}[\textbf{Proof}] Use identity $\mu(n)^2=\sum_{d^2\mid n}\mu(d)$, and other elementary routines, or confer to the literature.
\end{proof}

The constant coincides with the density of squarefree integers. Its approximate numerical value is
\begin{equation}\label{eq9339MN.72}
	\frac{6}{\pi^2}=\prod_{p\geq 2}\left ( 1-\frac{1}{p^2}\right )=0.607988295164627617135754\ldots,
\end{equation}
where $p\geq2$ ranges over the primes. The remainder term
\begin{equation}
	E(x)=\sum_{n <x} \mu^2(n) -\frac{6}{\pi^2}x
\end{equation} 
is a topic of current research, its optimum value is expected to satisfies the upper bound $E(x)=O(x^{1/4+\varepsilon})$ for any small number $\varepsilon>0$. Currently, $E(x)=O\left (x^{1/2}e^{-c\sqrt{\log x}}\right )$ is the best unconditional remainder term.\\

Assuming $t\ne0$, the earliest result for the autocorrelation of the squarefree indicator function $\mu^2(n)$ appears to be
\begin{equation}\label{eq8009MN.100}
	\sum_{n <x}\mu^2(n) \mu^2(n+t)=cx+O\left (x^{2/3} \right ),
\end{equation}
where $c>0$ is a constant, this is proved in \cite{ML1947}. Except for minor adjustments, the generalization to the $k$-tuple autocorrelation function has nearly the same structure.

\begin{thm}\label{thm8009MN.200}  Let $ q\ne0$, $a_0,a_1, \ldots,a_{k-1}$ be small integers, such that $0\leq a_0<a_1<\cdots<a_{k-1}$. Let $x\geq 1$ be a large number, and let $\mu: \mathbb{Z} \longrightarrow \{-1,0,1\}$ be the Mobius function. Then, 
	\begin{equation}\label{eq8009MN.200}
		\sum_{n <x}\mu^2(n+a_0)\mu^2(n+a_1)\cdots \mu^2(n+a_{k-1})=s_{k}x+O\left (x^{2/3+\varepsilon} \right ),\nonumber
	\end{equation}
	where the constant is given by the convergent product
	\begin{equation}\label{eq8009MN.210A}
		s_{k}=\prod_{p\geq 2}\left ( 1-\frac{\varpi(p)}{p^2}\right )>0,
	\end{equation}
	and
	\begin{equation}\label{eq8009MN.210B}
		\varpi(p)=\#\{m\leq p^2: qm+a_i\equiv 0 \bmod p^2 \text{ for } i=0,1,2, ..., k-1\}.
	\end{equation}
	The small number $\varepsilon>0$ and the implied constant depends on $q\ne0$.
\end{thm}
\begin{proof}[\textbf{Proof}]Consult \cite{ML1947}, \cite[Theorem 1.2]{MI2017}, and the literature for additional details.
\end{proof}

\begin{thm}\label{thm8009MN.350}  Let $x\geq 1$ be a large number, and let $\mu: \mathbb{Z} \longrightarrow \{-1,0,1\}$ be the Mobius function. If $t\ne0$ is a fixed integer,  then, 
	\begin{equation}\label{eq8009MN.350}
		\sum_{n \leq x}\mu(n)^2 \mu(n+t)=O\left (\frac{x}{(\log x)^c} \right ),\nonumber
	\end{equation}
	where $c> 0$ is an arbitrary constant.
\end{thm}
\begin{proof}[\textbf{Proof}] Substitute the identity $\mu(n)^2=\sum_{d^2\mid n}\mu(d)$, and switching the order of summation yield
	\begin{eqnarray}\label{eq8009MN.360B}
		\sum_{n \leq x}\mu(n)^2 \mu(n+t)&=&\sum_{n \leq x} \mu(n+t)\sum_{d^2\mid n}\mu(d)\\
		&=&\sum_{d^2 \leq x}\mu(d) \sum_{\substack{n\leq x\\d^2\mid n}}\mu(n+t)\nonumber\\
		&=&\sum_{d^2 \leq x^{2\varepsilon}}\mu(d) \sum_{\substack{n\leq x\\d^2\mid n}}\mu(n+t)+\sum_{x^{2\varepsilon}< d^2 \leq x}\mu(d) \sum_{\substack{n\leq x\\d^2\mid n}}\mu(n+t)\nonumber,
	\end{eqnarray}
	where $\varepsilon\in (0,1/4)$. Applying Corollary \ref{cor2225P.550} to the first subsum in the partition yields
	\begin{eqnarray}\label{eq8009MN.370B}
		\sum_{d^2 \leq x^{2\varepsilon}}\mu(d) \sum_{\substack{n\leq x\\d^2\mid n}}\mu(n+t)&\leq&\sum_{q \leq x^{\varepsilon}} \bigg |\mu(d) \sum_{\substack{n\leq x\\m\equiv b \bmod q}}\mu(m)\bigg |\\
		&=&O\left( \frac{x}{(\log x)^{c}}\right) \nonumber,
	\end{eqnarray}
	where $q=d^2$. An estimate of the second subsum in the partition yields
	\begin{eqnarray}\label{eq8009MN.380B}
		\sum_{x^{2\varepsilon}< d^2 \leq x}\mu(d) \sum_{\substack{n\leq x\\d^2\mid n}}\mu(n+t)&\leq&\sum_{x^{2\varepsilon}< d^2 \leq x} \sum_{\substack{n\leq x\\d^2\mid n}}1\\
		&\ll&x\sum_{d^2 \leq x}\frac{1}{d^2}\nonumber\\
		&\ll&x^{1-\varepsilon}\nonumber.
	\end{eqnarray}
	Summing \eqref{eq8009MN.370B} and \eqref{eq8009MN.380B} completes the verification.
\end{proof}

Given an integer $k<\log x$, the same technique can be extended to prove the more general nonlinear autocorrelation function
\begin{equation}\label{eq8009MN.400}
	\sum_{n \leq x}\mu(n+a_0)^{v_0}\mu(n+a_1)^{v_1}\cdots \mu(n+a_{k-1})^{v_{k-1}}=O\left (\frac{x}{(\log x)^c} \right ),
\end{equation}
where $a_0<a_1<\cdots<a_{k-1}$ is an integer $k$-tuple, and at least 1 of the integers $v_0, v_1, \ldots v_{k-1}\geq 0 $ is odd.

\section{Problems} \label{S6666P}

\subsection{Exponential Sums problems}
\begin{prob}\label{P6666.100} {\normalfont Let $x\geq x_0$ be a large number, and let $\alpha\ne0$ be a real number. Show that $$\sum_{ n<x}e^{i2 \pi n \alpha} \ll \min{1, ||\alpha||^{-1}}),$$
		where $||\alpha||=\min\{|n-\alpha|:n\geq0\}$. 
	}
\end{prob}

\subsection{Mobius autocorrelation function over the integers problems}

\begin{prob}\label{P6666.200MA} {\normalfont Let $x\geq x_0$ be a large integer, and let $a_0<a_1<a_{2}<x$ be an integer triple. Use exponential sums technique as in Theorem \ref{thmP1199.800} to compute a nontrivial estimate of the autocorrelation function 	$$\sum_{ n\leq x}\mu(n+a_0)\mu(n+a_1)\mu(n+a_{2})=O\left( \frac{x}{(\log x)^{c}}\right) .$$ 
	}
\end{prob}

\begin{prob}\label{P6666.200MF} {\normalfont Let $x\geq x_0$ be a large integer, and let $a_0<a_1<\cdots<a_{k-1}<x$ be an integer $k$-tuple. Determine the maximal length $k\overset{?}{\ll} \log x$ possible to have a nontrivial autocorrelation function 	$$\sum_{ n<x}\mu(n+a_0)\mu(n+a_1)\cdots \mu(n+a_{k-1})=O\left( \frac{x}{(\log x)^{c}}\right) .$$ 
	}
\end{prob}
\begin{prob}\label{P6666.200} {\normalfont Let $x\geq x_0$ be a large number, and let $a<b$ be small integers. Determine the oscillation results for the autocorrelation function 	$$\sum_{ n<x}\mu(n+a)\mu(n+b)=\Omega(x^{\beta})$$ over the integer. In \cite{KS2022}, there is some evidence and material for random autocorrelation functions.
	}
\end{prob}

\begin{prob}\label{P6666.215} {\normalfont Let $x\geq x_0$ be a large number, let $y\geq x^{1/2}$, and let $a<b$ be small integers. Determine a nontrivial upper bound for the autocorrelation function 
		$$\sum_{x\leq n<x+y}\mu(n+a)\mu(n+b)$$ over the integers in the short interval $[x,x+y]$.
	}
\end{prob}
\begin{prob}\label{P6666.220} {\normalfont Let $x\geq x_0$ be a large number, and let $a_0<a_1<a_2$ be small integers. Either conditionally or unconditionally, use sign patterns techniques to verify the asymptotic for the autocorrelation function $$\sum_{ n<x}\mu(n+a_0)\mu(n+a_1)\mu(n+a_2)=O\left( \frac{x}{(\log x)^{c}}\right) $$ over the integers. 
	}
\end{prob}

\begin{prob}\label{P6666.230} {\normalfont Let $x\geq x_0$ be a large number, let $1\leq a<q\ll (\log x)^c$, with $c\geq0$, and let $a_0<a_1$ be small integers. Determine a nontrivial upper bound for the autocorrelation function 
		$$\sum_{\substack{x\leq n\\n\equiv a \bmod q}}\mu(n+a_0)\mu(n+a_1)$$ over the integers in arithmetic progression-.
	}
\end{prob}

\subsection{Mobius autocorrelation function over the shifted primes problems}
\begin{prob}\label{P6666.200b} {\normalfont Let $x\geq x_0$ be a large number, and let $a<b$ be small integers. Determine the oscillation results for the autocorrelation function
		$$\sum_{ p<x}\mu(p+a)\mu(p+b)=\Omega(x^{\beta})$$ over the shifted primes, where  $\beta \in (1/2,1)$. In \cite{KS2022}, there is some evidence and material for random autocorrelation functions.
	}
\end{prob}
\begin{prob}\label{P6666.205} {\normalfont Let $x\geq x_0$ be a large number, and let $a<b$ be small integers. Assume the RH. Determine the conditional results for the autocorrelation function $$\sum_{ p<x}\mu(p+a)\mu(p+b)=O(x^{\beta})$$ over the shifted primes.
	}
\end{prob}

\begin{prob}\label{P6666.205b} {\normalfont Let $x\geq x_0$ be a large number, and let $a<b$ be small integers. Assume the RH. Determine the conditional results for the autocorrelation function
		$$\sum_{ p<x}\mu(p+a)\mu(p+b)=O(x^{\beta})$$ over the shifted primes, where  $\beta \in (1/2,1)$.
	}
\end{prob}

\begin{prob}\label{P6666.215b} {\normalfont Let $x\geq x_0$ be a large number, let $y\geq x^{1/2}$, and let $a<b$ be small integers. Determine a nontrivial upper bound for the autocorrelation function$$\sum_{x\leq p<x+y}\mu(p+a)\mu(p+b)$$ over the shifted primes over the short interval $[x,x+y]$.
	}
\end{prob}

\begin{prob}\label{P6666.225b} {\normalfont Let $a_0<a_1<\cdots<a_{k-1}$ be a subset of small integers. Determine the maximal value of $k\overset{?}{\ll} \log x$ for which the autocorrelation
		$$\sum_{p<x}\mu(p+a_0)\mu(p+a_1)\cdots\mu(p+a_{k-1})$$ over the shifted primes,	have nontrivial upper bounds.
	}
\end{prob} 

\begin{prob}\label{P6666.230b} {\normalfont Let $x\geq x_0$ be a large number, let $1\leq a<q\ll (\log x)^c$, with $c\geq0$, and let $a_0<a_1$ be small integers. Determine a nontrivial upper bound for the autocorrelation function $$\sum_{\substack{x\leq n\\n\equiv a \bmod q}}\mu(p+a_0)\mu(p+a_1)$$ over the shifted primes in arithmetic progression.
	}
\end{prob}

\subsection{Liouville autocorrelation function over the integers problems}

\begin{prob}\label{P6666.200B} {\normalfont Let $x\geq x_0$ be a large number, and let $a<b$ be small integers. Determine the oscillation results for the autocorrelation functions 	$$\sum_{ n<x}\lambda(n+a)\lambda(n+b)=\Omega(x^{\beta})$$ over the integers. In \cite{KS2022}, there is some evidence and material for random autocorrelation functions.
	}
\end{prob}
\begin{prob}\label{P6666.205B} {\normalfont Let $x\geq x_0$ be a large number, and let $a<b$ be small integers. Assume the RH. Determine the conditional results for the autocorrelation functions $$\sum_{ n<x}\lambda(n+a)\lambda(n+b)=O(x^{\beta})$$ over the integers, where  $\beta \in (1/2,1)$.
	}
\end{prob}

\begin{prob}\label{P6666.215B} {\normalfont Let $x\geq x_0$ be a large number, let $y\geq x^{1/2}$, and let $a<b$ be small integers. Determine a nontrivial upper bounds for the autocorrelation function 
		$$\sum_{x\leq n<x+y}\lambda(n+a)\lambda(n+b)$$ over the integers over the short interval $[x,x+y]$.
	}
\end{prob}

\begin{prob}\label{P6666.220B} {\normalfont Let $x\geq x_0$ be a large number, and let $a_0<a_1<a_2$ be small integers. Use sign patterns techniques to verify the asymptotic for the autocorrelation function $$\sum_{ n<x}\lambda(n+a_0)\lambda(n+a_1)\lambda(n+a_2)=O\left( xe^{-c\sqrt{\log x}}\right) $$ over the integers
	}
\end{prob}

\begin{prob}\label{P6666.225B} {\normalfont Let $a_0<a_1<\cdots<a_{k-1}$ be a subset of small integers. Determine the maximal value of $k\overset{?}{\ll} \log x$ for which the autocorrelation
		$$\sum_{n<x}\lambda(n+a_0)\lambda(n+a_1)\cdots\lambda(n+a_{k-1})$$ over the integers.
	}
\end{prob} 
\begin{prob}\label{P6666.230B} {\normalfont Let $x\geq x_0$ be a large number, let $1\leq a<q\ll (\log x)^c$, with $c\geq0$, and let $a_0<a_1$ be small integers. Determine a nontrivial upper bound for the autocorrelation function 
		$$\sum_{\substack{x\leq n\\n\equiv a \bmod q}}\lambda(n+a_0)\lambda(n+a_1)$$ over the integers in arithmetic progression.
	}
\end{prob}
\subsection{Liouville autocorrelation function over the shifted primes problems}
\begin{prob}\label{P6666.200a} {\normalfont Let $x\geq x_0$ be a large number, and let $a<b$ be small integers. Determine the oscillation results for the autocorrelation function
		$$\sum_{ p<x}\lambda(p+a)\lambda(p+b)=\Omega(x^{\beta})$$ over the shifted primes, where  $\beta \in (1/2,1)$. In \cite{KS2022}, there is some evidence and material for random autocorrelation functions.
	}
\end{prob}
\begin{prob}\label{P6666.205a} {\normalfont Let $x\geq x_0$ be a large number, and let $a<b$ be small integers. Assume the RH. Determine the conditional results for the autocorrelation function 
		$$\sum_{ p<x}\lambda(p+a)\lambda(p+b)=O(x^{\beta})$$ over the shifted primes, where  $\beta \in (1/2,1)$.
	}
\end{prob}

\begin{prob}\label{P6666.215a} {\normalfont Let $x\geq x_0$ be a large number, let $y\geq x^{1/2}$, and let $a<b$ be small integers. Determine a nontrivial upper bounds for the autocorrelation function $$\sum_{x\leq p<x+y}\lambda(p+a)\lambda(p+b)$$ over the shifted primes in the short interval $[x,x+y]$.
	}
\end{prob}

\begin{prob}\label{P6666.225a} {\normalfont Let $a_0<a_1<\cdots<a_{k-1}$ be a subset of small integers. Determine the maximal value of $k\overset{?}{\ll} \log x$ for which the autocorrelation
		$$\sum_{p<x}\lambda(p+a_0)\lambda(p+a_1)\cdots\lambda(p+a_{k-1})$$ over the shifted primes,	have nontrivial upper bounds.
	}
\end{prob}

\begin{prob}\label{P6666.230a} {\normalfont Let $x\geq x_0$ be a large number, let $1\leq a<q\ll (\log x)^c$, with $c\geq0$, and let $a_0<a_1$ be small integers. Determine a nontrivial upper bound for the autocorrelation function 
	$$\sum_{\substack{x\leq n\\n\equiv a \bmod q}}\lambda(p+a_0)\lambda(p+a_1)$$ over the shifted primes in arithmetic progression.
	}
\end{prob}




\begin{thebibliography}{998}

\bibitem{AT1976} Apostol, Tom M. \textit{\color{red}Introduction to analytic number theory}. Undergraduate Texts in Mathematics. Springer-Verlag, New York-Heidelberg, 1976.


\bibitem{BH1991} Baker, R. C., Harman, G. \textit{\color{red}Exponential sums formed with the Mobius function}. Journal of the London Math. Soc. (2) 43 (1991), 193-198.


\bibitem{CS1965} Chowla, S. \textit{\color{red}The Riemann hypothesis and Hilbert's tenth problem}. Mathematics and Its Applications, Vol. 4. Gordon and Breach Science Publishers, New York-London-Paris, 1965.




\bibitem{DH1937} Davenport, Harold. \textit{\color{red}On some series involving arithmetical functions. II.} Quart. J. Math. Oxf., 8:313-320, 1937.

\bibitem{DH2000}Davenport, H. \textit{\color{red}Multiplicative number theory.} volume 74 of Graduate Texts in Mathematics. Springer-Verlag, New York, third edition, 2000.

\bibitem{DL2012} De Koninck, Jean-Marie; Luca, Florian. \textit{\color{red}Analytic number theory. Exploring the anatomy of integers}. Graduate Studies in Mathematics, 134. American Mathematical Society, Providence, RI, 2012.

\bibitem{DLMF} NIST \textit{\color{red}Digital Library of Mathematical Functions}. http://dlmf.nist.gov/, 2019. F. W. J. Olver, ..., and M. A. McClain, eds.

\bibitem{GP1967} Gallagher, P. X. \textit{\color{red}The large sieve}. Mathematika 14 (1967), 14-20. 	


\bibitem{HA2013} Hildebrand, Adolph. \textit{\color{red}Introduction to Analytic Number Theory Math 531 Lecture Notes, Fall 2005.} http://www.math.uiuc.edu/~hildebr/ant.


\bibitem{HH2022} Helfgott, Harald Andres. \textit{\color{red}Expansion, divisibility and parity: an explanation}. http://arxiv.org/abs/2201.00799. 

\bibitem{HR2021} Helfgott, Harald Andres; Radziwill, Maksym.	\textit{\color{red}Expansion, divisibility and parity.} http://arxiv.org/abs/2103.06853.


\bibitem{HS1987} Hajela, D.; Smith, B. \textit{\color{red} On the maximum of an exponential sum of the Mobius function.} Lecture
Notes in Mathematics (Springer, Berlin, 1987) 145-164.




\bibitem{KS2022} Klurman, O.; Shkredov, I. D.; Xu, M. \textit{\color{red}On the random Chowla conjecture}. http://arxiv.org/abs/2202.08767.






\bibitem{MI2017} Mennema, Ingela. \textit{\color{red}The distribution of consecutive square-free numbers}. Master
Thesis, Leiden University, 2017.

\bibitem{ML1947} Mirsky, L. \textit{\color{red}Note on an asymptotic formula connected with $r$-free integers}. Quart. J. Math., Oxford Ser. 18, (1947). 178-182.	


\bibitem{MR2015} Matomaki, Kaisa, Radziwill, M., Tao, T. \textit{\color{red}An averaged form of Chowla's conjecture}. http://arxiv.org/abs/1503.05121. 


\bibitem{MS2002} Murty, R.; Sankaranarayanan, A.
\textit{\color{red}Averages of exponential twists of the Liouville function.} Forum
Mathematicum 14 (2002), 273-291. 
 



\bibitem{MV2007} Montgomery, Hugh L.; Vaughan, Robert C. \textit{\color{red}Multiplicative number theory. I. Classical theory}. Cambridge University Press, Cambridge, 2007.

\bibitem{MV1977} Montgomery, H., Vaughan, R.  C.\textit{\color{red}Exponential sums with multiplicative coefficients.} Inventiones Math. 43 (1977), 69-82.

 



\bibitem{RO2018} Ramare, Olivier. \textit{\color{red}Chowla's Conjecture: From the Liouville Function to the Moebius Function}. Part of the Lecture Notes in Mathematics book series (LNM,volume 2213), 16 June 2018.	







\bibitem{SW1971} Siebert, Hartmut; Wolke, D. \textit{\color{red}uber einige Analoga zum Bombierischen Primzahlsatz}. Math. Z.  122  (1971), no. 4, 327-341. 




\bibitem{TT2015} Tao, T. \textit{\color{red} The logarithmically averaged Chowla and Elliott conjectures for two-point correlations}. http://arxiv.org/abs/1509.05422.



\bibitem{WD1973} Wolke, Dieter. \textit{\color{red}uber die mittlere Verteilung der Werte zahlentheoretischer Funktionen auf Restklassen. I.} Math. Ann.  202  (1973), 1-25.

\end{thebibliography}
\end{document}